\documentclass{article}
\usepackage{amsfonts}
\usepackage[mathscr]{eucal}
 \usepackage{amsmath}
 \usepackage{amssymb}

\usepackage[T2A]{fontenc}
\usepackage[cp1251]{inputenc}
\usepackage[english]{babel}

\newtheorem{theorem}{Theorem}[section]
\newtheorem{lemma}[theorem]{Lemma}
\newtheorem{proposition}[theorem]{Proposition}
\newtheorem{definition}[theorem]{Definition}
\newtheorem{remark}[theorem]{Remark}
\newtheorem{corollary}[theorem]{Corollary}

\newcommand{\Rset}{\mathbb R}

\newcommand{\B}{\mathcal B}

\newcommand{\al}{{\alpha}}

\newcommand{\Om}{{\Omega}}

\newcommand{\si}{{\sigma}}
\newcommand{\va}{{\varphi}}
\newcommand{\vr}{{\varrho}}

\newcommand{\eps}{\varepsilon}
\newcommand{\ga}{{\gamma}}

\newcommand{\cA}{{\cal A}}
\newcommand{\cB}{{\cal B}}

\newcommand{\cF}{{\cal F}}

\newcommand{\cM}{{\cal M}}
\newcommand{\cN}{{\cal N}}

\newcommand{\cL}{{\cal L}}

\newcommand{\cW}{{\cal W}}

\newcommand{\sB}{\mathscr{B}}

\newcommand{\sR}{\mathscr{R}}

\newcommand{\sW}{\mathscr{W}}

\newcommand{\wla}{\widetilde{\lambda}}

\newcommand{\fA}{\mathfrak{A}}

\newcommand{\pd}{{\partial}}
\newcommand{\hf}{{\frac12}}

\newenvironment{declaration}[1]{\trivlist
\item[\hskip \labelsep{\bf #1 }]\ignorespaces}{\endtrivlist}
\newenvironment{proofof}[1]{\begin{declaration}{#1}}{\hfill
$\square$ \end{declaration}}
\newenvironment{proof}{\begin{proofof}{Proof.}}{\end{proofof}}

\begin{document}

\title{Dynamics of second order in time evolution equations \\ with state-dependent delay}
\author{Igor Chueshov$^{a,}$\footnote{\footnotesize Corresponding
author. E-mails:  chueshov@karazin.ua (I.Chueshov),
rezounenko@yahoo.com (A.Rezounenko)\ $^a$Department of Mechanics and
Mathematics, \ Karazin Kharkov National University,  Kharkov, 61022,
Ukraine\ $^b$Institute of Information Theory and Automation, Academy
of Sciences of the Czech Republic, P.O.Box~18, 18208,~CR}
 ~~ and ~~ Alexander  Rezounenko$^{a,b}$
 }

  \maketitle
\begin{abstract}
We deal with a class of second order in time nonlinear evolution
equations with state-dependent delay.
This class covers several important PDE models arising in the theory of
nonlinear plates. Our first result states well-posedness in a
certain space of functions which are $C^1$ in time. In contrast with
the first order models with discrete state-dependent delay this
result does not require any compatibility conditions. The solutions
constructed generate a dynamical system in a $C^1$-type space over
delay time interval. Our next result shows that this dynamical
system possesses
 compact global and exponential attractors of finite fractal dimension.
 To obtain this result we adapt the recently developed method of quasi-stability estimates.

\end{abstract}

\section{Introduction}
Our main goal is to study well-posedness and  asymptotic dynamics of
second order in time  equations with delay  of the form
\begin{equation}\label{sdd-2nd-01}
\ddot u(t)+ k\dot u(t) +A u(t) + F(u(t)) + M(u_t) = 0, \quad t>0,
\end{equation}
in some Hilbert space $H$. Here the dot over an element means time
derivative, $A$ is  linear and $F(\cdot)$ is nonlinear operators,
$M(u_t)$ represents (nonlinear) delay effect in the dynamics. All
these objects will be specified later.
\par
The main model we keep in mind is a nonlinear plate equation
of the form
\begin{equation}\label{sdd-2nd-plate}
\pd_{tt} u(t,x)+ k\pd_t u(t,x) +\Delta^2 u(t,x) + F(u(t,x)) +a u(t-\tau[u(t)],x) = 0, ~~x\in\Om,~ t>0,
\end{equation}
in a smooth bounded domain $\Om\subset\Rset^2$ with some boundary
conditions on $\pd\Om$. Here $\tau$ is a mapping defined on
solutions with values in some interval $[0,h]$, $k$ and $a$ are
constants. We assume that the plate is placed on some foundation;
the term $a u(t-\tau[u(t)],x)$ models effect of the Winkler type
foundation (see \cite{Selva-1979,Vlasov-1966}) with delay responce.
The nonlinear force $F$ can be Kirchhoff, Berger, or von Karman type
(see Section \ref{sect:ex-plate}). Our abstract model covers also
wave equation with state-dependent delay (see the discussion in
Section \ref{sect:ex-wave}).
\par
 We note that plate equations with
\emph{linear} delay terms were studied before mainly in  Hilbert
$L_2$-type
 spaces  on lag interval  (see, e.g.,
\cite{oldchueshov1,Chueshov-JSM-1992,CLW-delay,Cras-1995} and the
references therein). However this $L_2$-type situation does not
cover satisfactory the case of the state-depended delay of the form
described above. The point is that in this case the delay term in
\eqref{sdd-2nd-plate} is not even locally Lipschitz and thus
difficulties related to uniqueness may arise. The desire to have
Lipschitz property for this type delay terms leads naturally to
$C$-type  spaces which are not even reflexive. This provides us with
additional difficulties in contrast with the general theory
well-developed for second order in time equations in the Hilbert
space setting, see, e.g., \cite{Chueshov-Lasiecka-MemAMS-2008_book}
and also the literature cited there. In particular, in contrast with
the non-delayed  case (see
\cite{Chueshov-Lasiecka-MemAMS-2008_book,Chueshov-Lasiecka-2010_book,CL-hcdte-notes}),
in order to prove asymptotic smoothness of the flow (it is required
for the existence of a global attractor) we are enforced to assume
that the nonlinearity $F$ is either subcritical (in the sense
\cite{Chueshov-Lasiecka-MemAMS-2008_book}) of else the damping
coefficient $k$ in \eqref{sdd-2nd-01} is large enough. The main
reason for this is that we are not able to apply Khanmamedov's or
Ball's methods (see a discussion of both methods and the references
in \cite{CL-hcdte-notes}). The point is that we cannot guarantee
uniform in $t$ weak continuity in the phase space of the
corresponding functionals. Another reason is that the delay term
destroys the gradient structure of the model in the case of
potential nonlinearities $F$.
\par
The studies of    state-dependent delay models have a long history.
As it is mentioned in \cite{Hartung-Krisztin-Walther-Wu-2006}, early
discussion of differential equations with such a  delay goes back to
1806 when Poisson studied a geometrical problem. Since that time
many problems, initially described by differential equations without
delay or with constant delay, have been reformulated as equations
with state-dependent delay. It seems rather natural because many
models describing real world phenomena  depend on the past states of
the system.  Moreover,  it appears that  in many problems  the
constancy of the time delay is just an extra assumption which makes
the study  easier. The waiver of this assumption is naturally lead
to more realistic models and simultaneously makes analysis more
difficult.
 The general theory of (ordinary) differential equations with state-dependent delay
 is developed only recently (see. e.g.,
 \cite{Krisztin-Arino_JDDE-2001,Mallet-Paret,Walther_JDE-2003} and also the survey
 \cite{Hartung-Krisztin-Walther-Wu-2006} and the references therein). This theory
 essentially differs from that of constant or time-dependent
 delays  (see the references above and also Remark~\ref{re:M-nonLip} below).
\par
As for partial differential
 equations (PDEs) with delay their  investigation  requires the combination of both theories,
  methods and machineries (PDEs and delayed ODEs).
  The general theory of delayed PDEs    was started  with
  \cite{Fitzgibbon-JDE-1978,Travis-Webb_TAMS-1974} on the abstract level and
  was developed in last decades mainly for parabolic   type models
with constant and time-dependent delays (see e.g., the monographs
\cite{Wu-book} and the survey \cite{Ruess-1996}).  Abstract
approaches for $C$-type
\cite{Fitzgibbon-JDE-1978,Travis-Webb_TAMS-1974} and  $L_p$-type
\cite{Kunisch-Schappacher-JDE-1983} phase spaces are available.
 Partial differential
 equations with state-dependent delay are essentially
 less investigated,  see the discussion in the papers
 \cite{Rezounenko_JMAA-2007,Rezounenko_NA-2009} devoted to the parabolic case.
Some results (mainly, the existence and uniqueness) for the second
order in time PDEs with constant  delay are also available. They are
based on a reformulation    of the problem as a first order system
and application of the theory of such systems (see, e.g.,
\cite{Fitzgibbon-JDE-1978}). We also use this idea to get a local
existence and uniqueness for problem \eqref{sdd-2nd-01}. However to
the best of our knowledge, well-posedness  and asymptotic dynamics
of second order in time partial differential equations with
state-dependent delay have not been studied before.
\par
In our  approach we employ the special structure of second order in
time systems to get a globally  well-posed initial value problem for
mild solutions. As a phase space we choose some space of $C^1$-type
functions. The solutions we deal with are also $C^1$ functions. To
construct them we rewrite the second order in time equation (for
unknown $u(t)$) as a first order system (for unknown vector
$(u(t),\dot u(t))$) and look for continuous (mild) solutions to the
system. However in contrast with  approaches based on the
general theory (see, e.g.,  \cite{Fitzgibbon-JDE-1978} and   also
\cite[Section 3]{Walther_JDE-2003} and \cite[Section
2]{Hartung-Krisztin-Walther-Wu-2006}) we take into account natural
``displacement-velocity" compatibility from the very beginning at
the level of the phase space.
The solutions constructed have the
desired Lipschitz (even $C^1$ in time) property for the first
coordinate $u(t)$. In a sense it is an intermediate case between two
standard classes of merely continuous (mild) and $C^1$ (classical)
solutions $(u(t),v(t)),\, t\in [-h,T), T>0$ for a general first
order in time system  with delay:
$$\left\{
\begin{array}{c}
  \dot u(t) = {\cal F} (u_t,v_t), \\
  \dot v(t) = {\cal G} (u_t,v_t). \\
\end{array}
\right.
$$
 We  emphasize that due to the structure of our problem we do not
need any nonlinear compatibility type relations involving the right hand sides
of equations  which usually arise for
general first order (even, finite-dimensional) systems when $C^1$
solutions are studied (see \cite{Walther_JDE-2003} and also the survey
\cite{Hartung-Krisztin-Walther-Wu-2006}). We also refer to Section~\ref{ode}
below for a discussion of other  features of our approach.
\par

Our main result states that the dynamical system generated by
\eqref{sdd-2nd-01} in the space $W$ (see (\ref{sdd-2nd-05}) below)
of $C^1$ functions on the delay time interval possesses a compact
global attractor of finite fractal dimension. To achieve this result
we involve the method of quasi-stability estimates suggested in
 \cite{chlJDE04} and developed in \cite{Chueshov-Lasiecka-MemAMS-2008_book,Chueshov-Lasiecka-2010_book},
 see  also the recent survey in \cite{CL-hcdte-notes}.
However owing to the structure of the phase space we cannot apply directly the results known
for abstract quasi-stable systems and thus we are enforced to reconstruct the corresponding
argument in our state-dependent delay case.

The paper is organized as follows. In Section \ref{sect:wp} we
introduce our basic hypotheses and prove a well-posedness result.
Further sections are devoted to long-time dynamics. We first prove
that the system is dissipative (see Section~\ref{sect:dis}). In
Section~\ref{sect:qs} we show that the system satisfies some kind of
quasi-stability estimate on an invariant bounded absorbing set. This
allows us to establish the existence of compact finite-dimensional
global and exponential attractors in Section~\ref{sect:attr}. The
concluding Section~\ref{sect:ex} illustrates are main results by
applications to plate and wave models.

\section{Well-posedness and generation of  a dynamical system}\label{sect:wp}
\nopagebreak
The main outcome of this section is the fact that problem (\ref{sdd-2nd-01})
generates  dynamical system in an appropriate linear phase space of $C^1$ functions.

In our study we assume that:
\begin{itemize}
  \item [{\bf (A1)}]  \emph{In (\ref{sdd-2nd-01}),  $A$  is a positive operator
with a discrete spectrum in a separable Hilbert space $H$ with
domain $D(A)\subset H$.  Hence there exists an orthonormal basis $\{ e_k \}$
of $ H$ such that $$ A e_k = \mu_k e_k,\quad\mbox{with } 0<\mu_1\le
\mu_2\le \dots ,\quad \ \lim_{k \to \infty} \mu_k=\infty.$$ }
\end{itemize}
We can define the spaces $D(A^{\alpha})$ for $\al>0$ (see, e.g., \cite{Lions-Magenes-book}).  For
$h>0,$ we denote for short $C_{\alpha}=C([-h,0]; D(A^{\alpha}))$ which
is a Banach space with the following norm:
$$
\vert v
\vert_{C_{\alpha}} \equiv \sup \{\parallel   A^\alpha v(\theta
)\parallel : \theta \in [-h,0] \}.
$$
Here and below, $\parallel \cdot\parallel$ is the norm of $H$, and
$(\cdot , \cdot ) $ is the corresponding hermitian product.
We  also write $C=C_0$.
\begin{itemize}
  \item [{\bf (F1)}]\emph{ The nonlinear (non-delayed) mapping $F: D(A^{1\over 2}) \to H$
  is locally Lipschitz, i.e., for any $R>0$ there is $L_{F,R}>0$
such that for any $u^1,u^2$ with  $||A^{1\over 2}u^i||\le R$,
one has
 \begin{equation*}
 ||F(u^1)-F(u^2)|| \le L_{F,R} ||A^{1\over 2}(u^1-u^2)||.
\end{equation*}}
\end{itemize}

To describe the delay term $M$ we need the following standard notations from the theory of delay differential equations.
In (\ref{sdd-2nd-01}) and below, if $z$ is a continuous function from $\Rset$ into a space $Y,$ then as
in \cite{Hale,Wu-book}
$z_t{\equiv} z_t(\theta)\equiv z(t+\theta)$, $\theta\in [-h,0]$, denotes
the element of $C([-h,0];Y),$ while $h>0$ presents the (maximal) retardation time.
\par
In our considerations an  important role is played by the choice
of a phase space (see Remark~\ref{re:M-nonLip} below). We  use
the following one:
 \begin{equation}\label{sdd-2nd-05}
 W \equiv
 C([-h,0]; D(A^{1\over 2})) \cap C^1([-h,0]; H), 
\end{equation}
endowed with the norm $|\varphi|_W=|\varphi |_{C_{1/
2}}+|\dot\varphi |_{C_0}$
\par
We accept the following (basic) hypothesis concerning the delay term.
\begin{itemize}
  \item [{\bf (M1)}]\emph{
   The nonlinear delay term $M: W \mapsto H$
 is locally Lipschitz in the sense that
 \begin{equation*}
   \|M(\varphi^1)-M(\varphi^2)\|\le C_\vr \left[    |\varphi^1-\varphi^2|_{C_{1/2}}+
      |\dot\varphi^1-\dot\varphi^2|_{C_0}\right]
 \end{equation*}
 for every $\varphi^1, \varphi^2\in W$,   $|\varphi^j|_W\le\vr$, $j=1,2$.}
\end{itemize}

\begin{remark}\label{re:M-nonLip}{\rm
The main (benchmark) example\footnote{A more
general situation is described in hypothesis \textbf{(M3)} and Remark~\ref{re:M-tau} below.}
   of a  state-dependent delay term is
\begin{equation}\label{sdd-2nd-04}
 M(\varphi) =  \varphi(- \tau(\varphi)), \quad \varphi \in C, 
\end{equation}
where $\tau$ maps $C$ into some interval $[0,h]$. We notice that
this (discrete time) delay term $M$ is not locally Lipschitz in the
classical space of continuous functions $C=C([-h,0]; H)$, no matter
how smooth  the delay function $\tau : C\to [0,h]$ is. This may lead
to the non-uniqueness of solutions (see a discussion in the survey \cite{Hartung-Krisztin-Walther-Wu-2006}
 and the references wherein).
 This makes the study of
differential equations with state-dependent delays quite  different
from the one of equations with constant or time-dependent delays
\cite{Walther_book,Hale}. In such a situation the proof of the
well-posedness of a system requires additional efforts. For
instance, the main approach to $C^1$-solutions of general delay
equations is the so-called "solution manifold method"
\cite{Hartung-Krisztin-Walther-Wu-2006,Walther_JDE-2003} (see also
\cite{Rezounenko-Zagalak-DCDS-2013} for a parabolic PDE case) which
assumes some type of compatibility condition. It should be also
noted that there is an alternative approach avoiding  (nonlinear)
compatibility hypotheses. However it is
 based on an additional hypotheses concerning the delay mechanism
\cite{Rezounenko_NA-2009,Rezounenko_JMAA-2012}. Thus it is important
to deal with spaces in which we can guarantee a  Lipschitz property
for the mapping in  \eqref{sdd-2nd-04}. This is why to cover the
case we are enforced  to avoid  the space $C$ for the description of
initial data. For the same reason we cannot also use the idea
applied in  \cite{Kunisch-Schappacher-JDE-1983} and also in the
papers \cite{oldchueshov1,Chueshov-JSM-1992,CLW-delay,Cras-1995}
which deal with $L_2$-type spaces over the time delay interval. In
contrast, as we can see below the choice  of a Banach space of the
form \eqref{sdd-2nd-05} as a phase space allows us to guarantee
local Lipschitz property for the term in  \eqref{sdd-2nd-04}.
Moreover, this phase space takes into account  the natural
``displacement-velocity" relation from the very beginning. }
\end{remark}
\par
Thus bearing in mind the discussion above
we consider equation (\ref{sdd-2nd-01}) with the following initial
data
 \begin{equation}\label{sdd-2nd-IC}
 u_0=u_0(\theta )\equiv u(\theta ) = \varphi (\theta ), \quad\mbox{for}\; \theta \in [-h,0],
~~\va\in W.
\end{equation}

We can rewrite equation (\ref{sdd-2nd-01})  as the following first
order differential equation
 \begin{equation}\label{sdd-2nd-1st}
{d\over dt} U(t) + {\cal A} U(t) = {\cal N} (U_t), \quad t>0,~~
 \end{equation}
in the space $Y=D(A^{1/2})\times H$,
where $U(t)=(u(t); \dot u(t))$. 
Here the operator
${\cal A}$ and the map ${\cal N}$ are defined by
$$
{\cal A}U= (-v;
Au+ k v), \, \mbox{ for} \quad U=(u;v)\in D({\cal A})\equiv
D(A)\times D(A^{1/2})
$$
\begin{equation}\label{N-def}
{\cal N}(\Phi)= (0;F(\varphi(0))+M(\varphi)) \, \mbox{ for} \quad
\Phi=(\varphi;\dot \varphi), \, \varphi\in W.
\end{equation}
One can show
(see. e.g., \cite{Chueshov_Acta-1999_book})
 that the operator $\cA$ generates exponentially stable $C_0$-semigroup $e^{-\cA t}$ in $Y$.
\smallskip\par
The representation in  \eqref{sdd-2nd-1st} motivates the following definition.

\begin{definition}
  \label{de:mild} A {\it mild solution} of (\ref{sdd-2nd-01}),
(\ref{sdd-2nd-IC}) on an interval $[0,T]$ is defined as a function
\[
u \in C([-h,T];
D(A^{1/2})) \cap C^1([-h,T]; H),
\]
such that $u(\theta)=\varphi(\theta), \theta\in [-h,0]$ and  $U(t)\equiv (u(t);\dot u(t))$\footnote{
Below  $U(t)$ is also occasionally  called  by a mild solution.
} satisfies
\begin{equation}\label{sdd-2nd-05a}
 U(t)=  e^{-t{\cal A}}U(0) +\int^t_0 e^{-(t-s){\cal A}} {\cal N} (U_s)ds,\quad t\in [0,T].
\end{equation}
Similarly we can also define a mild solution on the semi-interval $[0,T)$.
\end{definition}

We can easily prove the following local result.

\begin{proposition}\label{pr:loc-exist}
Let (A1), (F1) and (M1) be valid. Then for any
$\varphi\in W$ there exist $T_\varphi>0$ and a unique mild solution
$U(t)\equiv (u(t);\dot u(t))$ of (\ref{sdd-2nd-01}),
(\ref{sdd-2nd-IC}) on the semi-interval interval $[0,T_\varphi)$. Solutions
continuously  depend on initial function $\varphi\in W$.
\end{proposition}

\begin{proof}
The argument for the local existence and uniqueness of a mild solution
 is standard (see, e.g., \cite{Fitzgibbon-JDE-1978}) and uses the Banach fixed point
 theorem for a contraction mapping in the space
 $C([-h,T]; D(A^{1/2})) \cap C^1([-h,T]; H)$
 with appropriately small $T$.
 \end{proof}
\par
To obtain a global well-posedness result we need additional hypotheses concerning $F$ and $M$.
 As in the case of the second order models without delay (see \cite{Chueshov-Lasiecka-MemAMS-2008_book} and \cite{Chueshov-Lasiecka-2010_book})
we use the following set of assumptions concerning  $F$.
\begin{itemize}
  \item [{\bf (F2)}]
{\it  The nonlinear mapping $F: D(A^{1\over 2}) \to H$
has the form
\begin{equation*}
 F(u) = \Pi^\prime(u) +F^{*}(u),
\end{equation*}
where $\Pi^\prime(u)$ denotes Fr\'{e}chet derivative\footnote{This means
that $\Pi^\prime(u)$    is an element in $D(A^{1\over 2})'$
such that  $|\Pi(u+v)-\Pi(u)-\langle \Pi^\prime(u),v\rangle |=o(\|A^{1/2}v\|)$
for every $v\in D(A^{1\over 2})$
}
of  a $C^1$-functional
$\Pi(u): D(A^{1\over 2}) \to R$  and
the mapping $F^{*}: D(A^{1\over 2}) \to H$  is globally Lipschitz, i.e.
\begin{equation}\label{sdd-2nd-07}
 ||F^{*}(u^1)- F^{*}(u^1)||^2 \le c_0 || A^{1\over 2} (u^1-u^2)||^2, \qquad u^1,u^2 \in D(A^{1\over
2}).
\end{equation}
Moreover, we assume that $\Pi(u)=\Pi_0(u) + \Pi_1(u)$, with
$\Pi_0(u)\ge 0$, $\Pi_0(u)$ is bounded on bounded sets in
$D(A^{1\over 2})$ and $\Pi_1(u)$ satisfies the property
\begin{equation}\label{sdd-2nd-08}
\forall\,\eta>0\; \exists\, C_\eta>0:~~
 |\Pi_1(u)| \le \eta \left( || A^{1\over 2}u||^2 + \Pi_0(u)\right) + C_\eta, \qquad u \in D(A^{1/2}).
\end{equation}
}
\end{itemize}
As it is well-documented in \cite{Chueshov-Lasiecka-MemAMS-2008_book,Chueshov-Lasiecka-2010_book}
the second order models with nonlinearities satisfying \textbf{(F2)} arises in many applications
(see also the discussion in Section~\ref{sect:ex}).
\par
We
assume also
\begin{itemize}
  \item [{\bf (M2)}]
 {\it  The nonlinear delay term $M: W \to H$ satisfies the linear growth condition:

\begin{equation}\label{sdd-2nd-13}
 ||M(\varphi)|| \le M_0+ M_1 \left\{ \max_{\theta\in [-h,0]} ||A^{1/2} \varphi(\theta)||+
  \max_{\theta\in [-h,0]} ||\dot \varphi(\theta)||\right\},~~\forall\, \va\in W,
\end{equation}
 for some  $M_j \ge 0$.
}
\end{itemize}
The main result of this section is the following assertion.
\begin{theorem}[Well-posedness]\label{th:well-pos}  Let (A1), (F1), (F2), (M1), and (M2) be valid. Then for any
$\varphi\in W$ there exists an unique global mild solution
$U(t)\equiv (u(t);\dot u(t))$ of (\ref{sdd-2nd-01}),
(\ref{sdd-2nd-IC}) on the interval $[0,+\infty)$. Solutions
satisfy an energy equality of the form
\begin{equation}\label{sdd-2nd-12}
{\cal E}(u(t),\dot u(t)) +k\int^t_0||\dot u(s)||^2 ds = {\cal
E}(u(0),\dot u(0))  - \int^t_0(F^{*}(u(s)),\dot u(s))\, ds -
\int^t_0(M(u_s),\dot u(s))\, ds.
\end{equation}
Here we denote
\begin{equation}\label{sdd-2nd-10}
 {\cal E}(u,v)\equiv E(u,v) + \Pi_1(u), \qquad E(u,v)\equiv {1\over 2}\left( ||v||^2 + || A^{1\over 2}u||^2 \right) +
 \Pi_0(u).
\end{equation}
Moreover, for any $\vr>0$ and $T>0$ there exists  $C_{\vr,T}$ such that
\begin{equation}\label{sdd-2nd-lip-sol}
\| A^{1/2}(u^1(t)-u^2(t))\|  + \| \dot u^1(t)-\dot u^2(t)\|
\le C_{\vr,T} |\va^1 -\va^2|_W,~~~ t\in [0,T],
\end{equation}
for any couple $u^1(t)$ and $u^2(t)$ of mild solutions with initial
data $\va^1$ and $\va^2$ such that  $ |\va^j|_W\le\vr$.
\end{theorem}
\begin{proof}
 The local existence and uniqueness of mild solutions are given by Proposition~\ref{pr:loc-exist}.
Let $ U=(u;\dot u)$   be a mild solution of (\ref{sdd-2nd-01}) and (\ref{sdd-2nd-IC}) on
the (maximal) semi-interval $[-h,T_\varphi)$ and
\[
f^u(t)\equiv F(u(t)) + M(u_t)\in C([0,T_\varphi);H).
\]
It is clear that we can consider $(u(t);\dot u(t))$ as a
 mild solution of the linear \emph{non-delayed} problem
\begin{equation}\label{sdd-2nd-10a}
\ddot v(t) +A v(t) + k\dot v(t) + f^u(t) = 0, \quad t\in
[0,T_\varphi),\quad (v(0);\dot v(0))=(\va (0); \dot\va (0)) \in Y.
\end{equation}
Therefore (see, e.g., \cite{Chueshov_Acta-1999_book}) one can see that $u(t)$ satisfies
the energy relation of the form
\begin{equation}\label{sdd-2nd-12lin}
E_0(u(t),\dot u(t)) +k\int^t_0||\dot u(s)||^2 ds = E_0(u(0),\dot u(0))  -
 \int^t_0(f^u(s),\dot u(s))\, ds,~~~ t<T_\va,
\end{equation}
where $ E_0(u,v)=\hf\left( \|A^{1/2}u\|^2+ \|v\|^2\right)$.
Using the structure of $f^u$ after some calculations (firstly performed on smooth functions)
we can show that
\[
 \int^t_0(f^u(s),\dot u(s))\, ds= \Pi(u(t))-\Pi(u(0)) + \int^t_0(F^*(u(s))+M(u_s),\dot u(s))\, ds.
\]
Therefore \eqref{sdd-2nd-12lin} yields \eqref{sdd-2nd-12} for every $t<T_\va$.
\par
By (\ref{sdd-2nd-07}) we have that $||F^{*}(u) ||\le
\sqrt{c_0}||A^{1/2}u|| +||F^{*}(0)||$. Therefore using
(\ref{sdd-2nd-12}) and (\ref{sdd-2nd-13})
we obtain that
\begin{align}\label{sdd-2nd-14}
  {\cal E}(u(t),\dot u(t)) +{k\over 2}\int^t_0||\dot u(s)||^2 ds
 \le &
{\cal E}(u(0),\dot u(0)) +  c_1 \int^t_0
(1+ ||A^{1\over 2}u(s)||^2)\, ds
\\ &
 +  c_2
\int^t_0\left[\max_{\theta\in [-h,0]} ||A^{1\over 2}
u(s+\theta)||^2+\max_{\theta\in [-h,0]} ||
\dot u(s+\theta)||^2\right] \, ds.\notag
\end{align}
One can  see that
\begin{align}\label{sdd-2nd-15}
\max_{\theta\in [-h,0]} ||A^{1\over 2}
u(s+\theta)||^2+\max_{\theta\in [-h,0]} || \dot u(s+\theta)||^2 \le
|\va |_{W}^2+ 2 \max_{\sigma\in [0,s]} E(u(\sigma),\dot u(\sigma))
\end{align}
for every $s\in [0,T_\va)$.
It follows from
 (\ref{sdd-2nd-08}) that  there exists  a
constant $c>0$ such that
\begin{equation}\label{sdd-2nd-11}
 {1\over 2} E(u,v) - c \le {\cal E}(u,v) \le 2E(u,v) +c, ~~ u\in D(A^{1\over
2}), v\in H.
\end{equation}
Therefore
 we use  (\ref{sdd-2nd-11}) and (\ref{sdd-2nd-15}) to
 continue (see (\ref{sdd-2nd-14})) as follows
\begin{equation*}
\max_{\sigma\in [0,t]} E(u(\sigma),\dot u(\sigma)) \le c \left(
1+t+E(u(0),\dot u(0))+ t\cdot |\va|_W^2 + \int^t_0 \max_{\sigma\in [0,s]}
E(u(\sigma),\dot u(\sigma))\, ds\right).
\end{equation*}
The application of Gronwall's lemma (to the
function $p(t)\equiv \max_{\sigma\in [0,t]} E(u(\sigma),\dot
u(\sigma))$ yields the
following  (a priori) estimate
\begin{equation*}
\max_{\sigma\in [0,t]} E(u(\sigma),\dot u(\sigma)) \le C \left(
1+E(u(0),\dot u(0))+ |\va|_W^2 \right)\cdot e^{a t}, \quad a>0,~~ t<T_\va,
\end{equation*}
which allows us in the standard way to extend the solution on the semi-axis $\Rset_+$.
\par
To prove \eqref{sdd-2nd-lip-sol} we use the fact that the difference $u(t)=u^1(t)-u^2(t)$
solves the problem in \eqref{sdd-2nd-10a} with
\[
f^u(t)=F(u^1(t)) + M(u^1_t)   -  F(u^2(t)) - M(u^2_t).
\]
This  completes the proof of Theorem~\ref{th:well-pos}.
\end{proof}

Using Theorem~\ref{th:well-pos}
we can define
an \textbf{ evolution operator} $S_t : W \to W$ for all $t\ge 0$ by the
formula $S_t \varphi = u_t,$ where $u(t)$ is
the mild solution  of (\ref{sdd-2nd-01}),
(\ref{sdd-2nd-IC}), satisfying  $u_0=\varphi$.
This operator satisfies the semigroup property
and generates a dynamical system $(S_t;W)$
with the phase space $W$
 defined in (\ref{sdd-2nd-05}) (for the
definition and more on dynamical systems see, e.g.,
\cite{Babin-Vishik,Chueshov_Acta-1999_book,Temam_book}).

\begin{remark}\label{re:evol}
{\rm We can equivalently define the dynamical system on the linear
space of vector-functions $\widetilde{W}\equiv\{ \Phi=\left(
\varphi; \dot \varphi\right) \, | \, \varphi\in W \}\subset
C([-h,0]; D(A^{1\over 2})\times H).$ In this notations evolution
operator reads $\widetilde{S}_t \Phi \equiv U_t$ and we have $W
\backepsilon \varphi \stackrel{G}{\longmapsto} \left( \varphi; \dot
\varphi\right) \in \widetilde{W}$ satisfying $G S_t =
\widetilde{S}_t G$.
In fact we already have used this observation in Definition~\ref{de:mild}
and Proposition~\ref{pr:loc-exist}.
}
\end{remark}

We conclude this section with a discussion of the existence  of
smooth solutions to problem  (\ref{sdd-2nd-01})  and
(\ref{sdd-2nd-IC}). In the following assertion we show that under
additional hypotheses mild solutions become strong.
\begin{corollary}[Smoothness]\label{co:smooth}
 Let the hypotheses of  Theorem~\ref{th:well-pos}  be in force
 with assumption \textbf{(M1)} in  the following (stronger)
 form
 \begin{equation}\label{M-lip-str}
   \|M(\varphi^1)-M(\varphi^2)\|\le C_\vr     |\varphi^1-\varphi^2|_{C_{0}}
 \end{equation}
 for every $\varphi^1, \varphi^2\in W$,   $|\varphi^j|_W\le\vr$, $j=1,2$.
 If the initial function $\va(\theta)$ possesses the property
 \begin{equation}\label{smo-ini}
 \va(0)\in D(A),~~~ \dot\va(0)\in  D(A^{1/2}),
 \end{equation}
 then the solution $u(t)$ satisfies the relations
 \begin{equation}\label{smo-sol}
    u(t)\in L_\infty(0,T;  D(A)),  ~~ \dot u(t)\in L_\infty(0,T;  D(A^{1/2})), ~~
     \ddot u(t)\in L_\infty(0,T; H)
 \end{equation}
 for every $T>0$.
 If in addition $F(u)$ is  Fr\'{e}chet differentiable and
 $\|F'(u) v\|\le C_r\|A^{1/2}v\|$ for every $ u\in D(A)$ with $\|Au\|\le r$,
 then we have
 \begin{equation}\label{smo-sol-C}
    u(t)\in C(\Rset_+;  D(A)),  ~~ \dot u(t)\in C(\Rset_+;  D(A^{1/2})), ~~
     \ddot u(t)\in C(\Rset_+; H).
 \end{equation}
\end{corollary}

\begin{proof}
Let $u(t)$ be a solution.
By Theorem~\ref{th:well-pos} we have that
\[
\max_{[-h,T]}\left( \|A^{1/2} u(t)\|^2 +\| \dot u(t)\|^2\right)\le R_T
\]
for some $R_T$. Now we
 note that under condition
\eqref{M-lip-str} the function $t\mapsto f(t)\equiv M(u_t)$ is
Lipschitz on any interval $[0,T]$ with values in $H$. Indeed, by \eqref{M-lip-str} we have that
\[
 \|M(u_{t_1})-M(u_{t_2})\|\le C_{R_T} \max_{[-h,0]}\Big\|\int_{t_2+\theta}^{t_1+\theta}
 \dot u(\xi)d \xi\Big\|\le  C_{R_T} R_T|t_1-t_2|.
 \]
Thus  the
derivative $\dot f(t)$ (in the
sense of distributions) is bounded in $H$. This allows us to apply Theorem 2.3.8~\cite[p.63]{Chueshov-Lasiecka-2010_book}
(see also \cite[Chapter 4]{showalter})
to obtain
the conclusion in \eqref{smo-sol}.
\par
Property \eqref{smo-sol-C} follows from
 \cite[Proposition 2.4.37]{Chueshov-Lasiecka-2010_book}.
\end{proof}

\begin{remark}\label{re:smooth}
{\rm The property in  (\ref{M-lip-str}) means that $M$ is Lipschitz
on subsets in $C=C([-h,0]; H)$ which are bounded in $W$. Following
\cite[Definition 1.1, p.106]{Mallet-Paret} we call this property as
"locally almost Lipschitz" on $C$.  It is also remarkable that in
order to obtain strong solutions we need to assume an additional
smoothness of initial data in the right end point of the interval
$[-h,0]$  only (see \eqref{smo-ini}). A similar effect was observed
earlier in  \cite{Rezounenko_NA-2010,Rezounenko-Zagalak-DCDS-2013}
in the context of parabolic PDEs with discrete state-dependent
delay.
\par
We also note that under conditions of Corollary~\ref{co:smooth} with differentiable $F$
we have that solutions are $C^2$ on the semi-axis $\Rset_+$ with values in $H$,
and in $C^1$ on the extended semi-axis $[-h,+\infty)$.
Assuming the smoothness of the initial data $\va$ and
some compatibility conditions we can show that the solutions are $C^2$-smooth on  $[-h,+\infty)$.
More precisely, if we assume that
\begin{equation}\label{Wsm}
  \va\in  W_{sm}= C^2([-h,0]; H)\cap C^1([-h,0]; D(A^{1/2}))\cap C([-h,0]; D(A)),
\end{equation}
 then the solution $u$ possesses the property in \eqref{smo-sol-C}
 with $[-h,+\infty)$ instead of $\Rset_+$ if and only if this
 smoothness property holds in the zero moment. The later property is obviously valid if and only if
 we have the following compatibility condition
\begin{equation}\label{comp-C}
       \ddot{\va}(0) + k\dot\va (0) +A \va (0) + F(\va (0)) + M(\va) = 0.
\end{equation}
Moreover, one can see that the set
\begin{equation}\label{L-smoth}
\cL=\left\{ \va\in W_{sm}\, : \va ~~\mbox{satisfies (\ref{comp-C})}\right\} \subset W.
\end{equation}
is forward invariant with respect to the flow
$S_t$, i.e., $S_t\cL\subset\cL$ for all $t>0$.
Thus the dynamics is defined  in smother spaces.
 The set $\cL$ is an
analog to the solution manifold used in \cite{Walther_JDE-2003} for
the ODE case and in \cite{Rezounenko-Zagalak-DCDS-2013} for the  parabolic
PDE case
as a well-posedness class.
 }
\end{remark}

\section{Asymptotic properties: dissipativity}\label{sect:dis}
Now we start to study the long-time dynamics  of the system
$(S_t,W)$ generated by mild solutions to problem \eqref{sdd-2nd-01}.
For this we need to impose additional hypotheses. In analogy with
\cite{Chueshov-Lasiecka-MemAMS-2008_book} and  \cite[Chapter
8]{Chueshov-Lasiecka-2010_book} concerning the nonlinear
(non-delayed) term $F$ we assume
\begin{itemize}
  \item[{\bf (F3)}]
  {\it
 The nonlinear term $F: D(A^{1\over 2}) \to H$ (see (F2) above for notations) satisfies
\par
{\bf (a)}
 there are constants $\eta\in [0,1), c_4,c_5>0$
such that
\begin{equation}\label{sdd-2nd-19}
-(u,F(u))\le \eta ||A^{1\over 2}u||^2 -c_4 \Pi_0(u) +c_5,\quad u\in
D(A^{1\over 2});
\end{equation}
{\bf (b)}
for every $\widetilde{\eta}>0$ there exists $C_{\widetilde{\eta}}>0$
such that
\begin{equation}\label{sdd-2nd-20}
||u||^2\le C_{\widetilde{\eta}} +\widetilde{\eta}\left( ||A^{1\over
2}u||^2 +\Pi_0(u) \right), \quad u\in D(A^{1\over 2});
\end{equation}
{\bf (c)} the non-conservative term $F^{*}$ satisfies the subcritical linear growth condition,
i.e., there exist $\hat\delta>0$, $c_6, c_7\ge 0$ such that
\begin{equation}\label{sdd-2nd-21}
||F^{*}(u)||^2 \le c_6 +c_7 ||A^{{1\over 2}-\hat\delta}u||^2 \quad
\hbox{for any}\quad u\in D(A^{1\over 2}).
\end{equation}
}
\end{itemize}
As for the delay term, we concentrate on the case of
\emph{ discrete state-dependent delay} and impose the following hypothesis.
\begin{itemize}
  \item[{\bf (M3)}] {\it  The nonlinear delay term $M: W \mapsto H$
  has the form $ M(u_t)=  G(u(t-\tau(u_t)))$,
where $\tau$ maps $W$ into the interval $[0,h]$ and $G$ is a
globally Lipschitz mapping from $L_2(\Om)$ into itself.
}\end{itemize}
\begin{remark}\label{re:M-tau}
{\rm Since the term $M(u_t)$  satisfying  {\bf (M3)} can be  written in the form
\begin{equation}\label{sdd-2nd-26}
 M(u_t)=  G(u(t-\tau(u_t)))\equiv G\left( u (t) - \int^t_{t-\tau(u_t)} \dot u(s)\, ds\right),
\end{equation}
we have that
\begin{equation*}
||M(u_t)||  \le ||G(0)||+L_G\left[ ||u(t))|| +
\int^t_{t-h} || \dot u(s)||\, ds\right],
\end{equation*}
where  $L_G$ is the Lipschitz constant of the mapping $G$. This yields that
\begin{equation}\label{M-bound}
||M(u_t)||^2\le g_0+ g_1 ||u(t))||^2 +g_2(h)
\int^t_{t-h} || \dot u(s)||^2\, ds
\end{equation}
with $g_0=4||G(0)||^2$, $g_1=4 L_G^2$ and $g_2(h)=2 L_G^2 h$.
Thus {\bf (M3)} implies {\bf (M2)}. To guarantee {\bf (M1)}
we need to assume that
 $\tau$
 is locally Lipschitz on $W$:
 \begin{equation*}
   |\tau(\varphi^1)-\tau(\varphi^2)|\le C_\vr \left[    |\varphi^1-\varphi^2|_{C_{1/2}}+
      |\dot\varphi^1-\dot\varphi^2|_{C_0}\right]
 \end{equation*}
 for every $\varphi^1, \varphi^2\in W$,   $|\varphi^j|_W\le\vr$, $j=1,2$.
Indeed, from  \eqref{sdd-2nd-26} we have that
\begin{align*}
||M(u^1_s)-M(u^2_s)|| \le & L_G ||u^1(s-\tau(u^1_s))-
u^1(s-\tau(u^2_s))|| + L_G ||u^1(s-\tau(u^2_s))-
u^2(s-\tau(u^2_s))||\notag
\\
\le & \vr L_G  |\tau(u^1_s)-\tau(u^2_s)|+L_G  \max_{\theta\in
[-h,0]}||u^1(s+\theta)-u^2(s+\theta)||
 \notag \\
 \le  & (1+\vr C_\vr)L_G
|u^1_s-u^2_s|_W
\end{align*}
for all $u^1_s, u^2_s\in W$,   $|u^j_s|_W\le\vr$, $j=1,2$.
 Instead of the structure presented in \textbf{(M3)} we can also take
 a delay term of the form
 \[
 M(u_t)= \sum_{k=1}^N G_k(u(t-\tau_k(u_t))),
 \]
 or even consider  an integral version of this sum.
 Moreover instead of \textbf{(M3)} we can postulate the property  in \eqref{M-bound}
 with the constants $g_0$, $g_1$ independent of $h$ and $g_2(h)\to 0$ as $h\to 0$.
  }
\end{remark}

Our first step in the study of qualitative behavior  of the system $(S_t,W)$ is the following  (ultimate)
dissipativity property.

\begin{proposition}\label{pr:diss}
   Let assumptions (A1), (F1), (F2), (F3), (M1) and (M3) be valid.
Then for  any $k_0$ there exists $h_0=h(k_0)>0$ such that
for every $(k,h)\in [k_0,+\infty)\times (0,h_0]$
  the system $(S_t,W)$ is dissipative, i.e., there exists $R>0$
such that for every $\vr>0$  we can find $t_\vr>0$ such that
\begin{equation*}
    |S_t \varphi|_W\le R~~\mbox{for all}~~ \quad \varphi\in W, \quad  ~|\varphi|_W\le \vr,\quad  ~t\ge t_\vr.
\end{equation*}
Moreover for every fixed $k_0>0$ the dissipativity radius $R$ is
independent of $k>k_0$ and the delay time $h\in (0,h_0]$. Thus the
dynamical system $(S_t,W)$ is dissipative (uniformly for $k>k_0$ and
$h\le h_0$).
\end{proposition}
\begin{remark}\label{re:dis}
{\rm \textbf{(1)}
 The dissipativity property can be written in the form
\begin{equation*}
||\dot u(t)||^2 + ||A^{1\over 2}u(t)||^2 \le R^2 \quad \hbox{ for
all } \quad t\ge t_\vr,
\end{equation*}
provided the initial function $\va\in W$ possesses the property $|\va|_W\le \vr$.
We can also show in the standard way (see, e.g., \cite{Chueshov_Acta-1999_book} or \cite{Temam_book})
that there exists  a bounded \emph{forward invariant} absorbing set
$\cB$ in $W$ which belongs to the ball $\{\va\in W\, :  |\va|_W\le R\}$ with the radius
$R$ independent of $k\in [k_0,+\infty)$.
\par
 \textbf{(2)}
As we see in the proof below the restriction on the delay time $h$
has the form $h\le\beta k_0$ for some $\beta>0$. Thus increasing the
low bound $k_0$ for the damping interval we can increase the
corresponding admissible interval for $h$. This fact is compatible
with observation that large time lag may destabilize the system.
For instance, it is known from \cite{Cooke-Grossman1982} that
for the delayed 1D ODE
\[
\ddot u(t)+k\dot u(t)+ a u(t) +  u(t-\tau)=0
\]
with $a>1$  and $2a>k^2$ there exist $0<\tau_*<\tau^*$ such that
the zero solution is  stable for all $\tau<\tau_*$
and unstable when  $\tau>\tau^*$. This example also demonstates
 the role of the large damping. Indeed, if $k^2>2a>2$, then
 (see \cite{Cooke-Grossman1982}) the zero solution is
 stable for \emph{all} $\tau\ge 0$.
Thus  large time delay requires  large damping coefficient to
stabilize the system.

}
\end{remark}
\begin{proof}
 We use the Lyapunov method  to get the
result. The presence of the delay term $M$ requires some
modifications of the standard functional $V$  usually of the second order systems
(see, e.g., the proof of Theorem 3.10
\cite[p.43-46]{Chueshov-Lasiecka-MemAMS-2008_book}).
\par
We use the following functional
\begin{equation*}
\widetilde{V}(t)\equiv {\cal E}(u(t),\dot u(t)) + \gamma(u(t),\dot
u(t))+ \frac{\mu}h \int_0^h \left\{ \int_{t-s}^t ||\dot u(\xi)||^2
d\, \xi \right\} \, ds.
\end{equation*}
Here ${\cal E}$ is defined in (\ref{sdd-2nd-10}) and the positive parameters
$\gamma$ and $\mu$ will be chosen later.
\par
The main idea behind inclusion of an additional delay term in $\widetilde{V}$  is to
find a compensator for $M(u_t)$. The compensator is determined by
the structure of the mapping $M$ (see \eqref{sdd-2nd-26}  and \eqref{M-bound}). This idea
was already applied in \cite[p.480]{Chueshov-Lasiecka-2010_book} and
\cite{CLW-delay} in the study of a flow-plate interaction model
which contains a linear constant delay term with the critical spatial
regularity. The corresponding compensator has a different form in
the latter case.
\par
One can see from \eqref{sdd-2nd-08} that there is $0<\ga_0<1$ such that
\begin{equation}\label{sdd-2nd-23a}
\hf E(u(t),\dot u(t)) -c\le
\widetilde{V}(t)\le 2 E(u(t),\dot u(t)) +  \mu \int_0^h ||\dot u(t-\xi)||^2 d\, \xi +c.
\end{equation}
for every $0<\ga\le \ga_0$, where $c$ does not depend on $k$.
\par
Let us consider the time
derivative of $\widetilde{V}$ along a solution.
One can easily check that
\begin{align}\label{sdd-2nd-24}
\frac{d}{dt} (u(t),\dot u(t))= \|\dot u(t)\|^2 -k(u(t),\dot u(t))  -
||A^{1\over 2}u(t)||^2 -(u,F(u)) - (u,M(u_t)).
\end{align}
 Combining (\ref{sdd-2nd-24}) with the energy relation in \eqref{sdd-2nd-12}
 and using the estimate  $k(u,\dot u)\le k^2\|\dot u\|^2+\frac14\|u\|^2$
   we get
\begin{align*}
\frac{d}{dt} \widetilde{V}(t) \le  & 
-(k -\gamma(1+k^2) -\mu) ||\dot u(t)||^2    -(F^{*}(u(t))+
M(u_t),\dot u(t)) \\ & - \gamma \left(-\frac14\|u(t)\|^2+
||A^{1\over 2}u(t)||^2 +(u,F(u))+ (u,M(u_t)) \right) -\frac{\mu}h
\int_0^h ||\dot u(t-\xi)||^2 d\, \xi. \notag
\end{align*}
Using (\ref{sdd-2nd-21}) we get
$$ |(F^{*}(u(t)),\dot u(t))| \le {1\over 8}k ||\dot u(t)||^2 +
{2\over k}||F^{*}(u(t)) ||^2 \le {1\over 8}k ||\dot u(t)||^2 +
{2c_6\over k} + {2c_7\over k}||A^{1/2-\delta} u(t)||^2.
$$
Hence using the inequality  $|(M(u_t),\dot u(t))| \le {1\over 8}k
||\dot u(t)||^2 + {2\over k}||M(u_t) ||^2$ and also estimate (\ref{M-bound}) we obtain that
\begin{align*}
-(F^{*}(u(t))+ M(u_t),\dot u(t)) \le & \frac14 k \|\dot u(t)\|^2
+\frac{c_0}k\left[1+\|A^{1/2-\delta} u(t)\|^2 + ||u(t)||^2 \right] \\ &
+\frac{g_2(h)}{k}\int_0^h ||\dot
u(t-\xi)||^2 d\, \xi, \notag
 \end{align*}
 where $c_0= 2 \max \{c_7; c_6+ g_0, g_1\}>0$ does not depend on $k$.
 \par
In a similar way (see (\ref{M-bound})) we also have that
\begin{align*}
|(u(t),M(u_t))|
\le g_2(h) \int^h_0 ||\dot u(t-\xi)||^2 d\, \xi + C(g_0,g_1)(1+ ||u(t)||^2).
\end{align*}
The relations in  \eqref{sdd-2nd-19} and  (\ref{sdd-2nd-20})
with small enough $\widetilde{\eta}>0$ yields
\[
 C(g_0,g_1)(1+ ||u||^2)  -||A^{1\over 2}u||^2 -(u,F(u))\le - 3 a_0
E(u, \dot u) + \|\dot u \|^2+ a_1
\]
for  some $a_i>0$.
Thus it follows from  the relations above that
\begin{align*}
\frac{d}{dt} \widetilde{V}(t) \le & 
-\left(\frac34 k -\gamma(2+k^2) -\mu \right) ||\dot u(t)||^2
+\frac{c_0}k\left[1+\|A^{1/2-\delta} u(t)\|^2+\|u(t)\|^2\right] \notag \\ & +
\gamma \left(- 3 a_0 E(u(t), \dot u(t)) + a_1  \right) + \left[
-\frac{\mu}h +\left(\frac{2}{k} +\gamma\right)g_2(h)\right] \int_0^h ||\dot u(t-\xi)||^2 d\, \xi .
\end{align*}
As in \cite[p.45]{Chueshov-Lasiecka-MemAMS-2008_book} using (\ref{sdd-2nd-20})  we can conclude
\[
\frac{c_0}k\left[ \|A^{1/2-\delta}u(t)\|^2+  \|u(t)\|^2\right]\le  \gamma a_0  E(u(t), \dot
u(t))+ \frac1k b\Big(\frac1{\ga k}\Big),
\]
where $b(s)$ is a non-decreasing function. Thus using
\eqref{sdd-2nd-23a} we arrive at the relation
\begin{align}\label{sdd-2nd-25b}
\frac{d}{dt} \widetilde{V}(t) +\ga a_0 \widetilde{V}(t)  \le &
-\left( \frac34 k -\gamma(2+k^2) -\mu \right) ||\dot u(t)||^2    +
\ga \left[ \tilde{a}+\frac1{\ga k}\tilde{b}\Big(\frac1{\ga k}\Big)\right],
\notag \\ & + \left[ -\frac{\mu}h +\mu\ga a_0+
\left(\frac{2}{k} +\gamma\right)g_2(h)\right] \int_0^h
||\dot u(t-\xi)||^2 d\, \xi .
\end{align}
Take $\mu =\frac{k}4$ and $\ga=\frac{\si k}{4+2k^2}$, where $0<\si<1$ is
chosen such that $\ga\le \ga_0$ for all $k>0$ (the bound $\ga_0$ arises in \eqref{sdd-2nd-23a}).
 Assume also that $h$ is such that
\begin{equation}\label{h-restr}
    -\frac{k}{4h} +\frac{\ga k}4 a_0+ \left(\frac{2}{k} +\gamma\right)g_2(h)\le 0.
\end{equation}
Then \eqref{sdd-2nd-25b} implies that
\begin{align}\label{sdd-2nd-25c}
\frac{d}{dt} \widetilde{V}(t) +\ga a_0 \widetilde{V}(t)  \le
\ga \left[ \tilde{a}+\frac1{\ga k}\tilde{b}\Big(\frac1{\ga k}\Big)\right],
\end{align}
One can see there is $\si_0=\si_0(k_0)$ such that  $\si_0\le \ga k \le \si/2$ for all $k\ge k_0$.
Therefore  from \eqref{sdd-2nd-25c}
we obtain that
\begin{align}\label{sdd-2nd-25d}
 \widetilde{V}(t) \le\widetilde{V}(0) e^{-\ga a_0 t} +
\frac{1}{a_0}(1-e^{-\ga a_0 t}) \left[ \tilde{a}+\frac1{\si_0}\tilde{b}\Big(\frac1{\si_0}\Big)\right],
\end{align}
provided
\begin{equation}\label{h-restr2}
    -\frac{k_0}{4h} +\frac18 a_0+ g_2(h)\left(\frac{2}{k_0}+\frac{1}{2}\right)\le 0.
\end{equation}
Here we used \eqref{h-restr} and properties $\gamma k<{1\over 2},
\gamma<{1\over 2}$ which follow from the choice of $\gamma$.
 One can see that there exists
$\beta>0$ such that \eqref{h-restr2} holds when $h\le \beta k_0$.
Under this condition relation \eqref{sdd-2nd-25d} implies the
desired (uniform in $k$) dissipativity property\footnote{In fact for this property we only need that
$g_2(h)\to0$  as $h\to 0$ in estimate  \eqref{M-bound}.
}
 and completes the proof of Proposition~\ref{pr:diss}.
\end{proof}

\section{Asymptotic properties: quasi-stability}\label{sect:qs}
In this section we show that  the system $(S_t,W)$ generated by the delay equation in \eqref{sdd-2nd-01}
possesses some asymptotic compactness property which is called "quasi-stability"
(see. e.g., \cite{Chueshov-Lasiecka-2010_book}  and \cite{CL-hcdte-notes})
and means that any two trajectories of the system are convergent modulo compact term.
As it was already seen at the level of non-delayed systems (see, e.g., \cite{Chueshov-Lasiecka-MemAMS-2008_book,Chueshov-Lasiecka-2010_book,CL-hcdte-notes}
and the references therein)
this property usually leads to several important conclusions concerning global long-time dynamics
of the system.
\par
Quasi-stability requires additional hypotheses concerning the system.
We assume
\begin{itemize}
  \item[{\bf (M4)}]
  {\it There exists $\delta>0$ such that  {\it the delay term $M$} 
satisfies subcritical local  Lipschitz property i.e.  for any $\vr>0$
there exists $L(\vr)>0$ such that for any $\varphi^i,
i=1,2$ such that  $||\varphi^i||_W\le \vr$, one has
\begin{equation}\label{sdd-2nd-35}
\|M(\varphi^1) - M(\varphi^2)\| \le L(\vr)
\max_{\theta\in [-h,0]} ||A^{1/2-\delta}(\varphi^1(\theta)-\varphi^2(\theta))||.
\end{equation}
}
\end{itemize}
As in Remark~\ref{re:M-tau} one can see that \eqref{sdd-2nd-35}
holds for $M$ given by \eqref{sdd-2nd-26} if we assume that
\begin{equation}\label{sdd-2nd-35a}
|\tau(\varphi^1) - \tau(\varphi^2)| \le L_\tau(\vr)
\max_{\theta\in [-h,0]} ||A^{1/2-\delta}(\varphi^1(\theta)-\varphi^2(\theta))||.
\end{equation}
Below we also distinguish the cases of critical and subcritical
(non-delayed) nonlinearities $F$. We introduce the following
hypothesis.

\medskip
\begin{itemize}
\item[{\bf (F4)}]
{\it
We assume that the nonlinear (non-delayed) mapping $F: D(A^{1\over 2}) \to H$  satisfies
one of the following conditions:
\begin{itemize}
  \item [{\bf (a)}]\textbf{either} it is \emph{subcritical}, i.e.,  there is positive $ \eta$ such that for any $R>0$ there exists
  $L_F(R)~>~0$
such that
 \begin{equation}\label{sdd-2nd-34}
 ||F(u^1)-F(u^2)|| \le L_F(R) ||A^{{1\over 2}-\eta}(u^1-u^2)||,~~\forall\, u^1,u^2\in D(A^{\hf}),~ ||A^{1\over 2}u^i||\le R;
\end{equation}
  \item[{\bf (b)}] \textbf{ or else} it is \emph{critical}, i.e., \eqref{sdd-2nd-34} holds
with $\eta=0$, and the damping parameter $k$ is large enough.
\end{itemize}
}
\end{itemize}

\begin{theorem}[Quasi-stability]\label{th:qs}
 Let assumptions (A1), (F1), (F2), (F4), (M1),
(M2) and (M4) be in force.
Then
there exists positive constants $C_1(R)$, $\widetilde{\lambda}$ and $C_2(R)$ such that
for any two solutions $u^i(t)$
with initial data $\va^i$ and
 possessing the properties
\begin{equation}\label{sdd-2nd-22a}
||\dot u^i(t)||^2 + ||A^{1\over 2}u^i(t)||^2 \le R^2 \quad \hbox{ for
all } \quad t\ge -h,~~ i=1,2,
\end{equation}
the following quasi-stability  estimate holds:
\begin{align}\label{qs-est}
||\dot u^1(t)-\dot u^2(t)||^2 + ||A^{1\over 2}(u^1(t)- u^2(t))||^2
\le & C_1(R)  e^{-\widetilde{\lambda} t} |\va^1-\va^2|^2_W \notag \\
& + C_2(R)\max_{\xi\in [0,t]}
||A^{{1/2}-\delta}(u^1(\xi)-u^2(\xi))||^2
\end{align}
with some $\delta>0$. In the critical case $k\ge k_0(R)$ for some $k_0(R)>0$.
\end{theorem}
We emphasize that Theorem~\ref{th:qs} does not assume  \textbf{(F3)} and \textbf{(M3)}
and deals only with a pairs of uniformly bounded solutions.
However, if the conditions in  \textbf{(F3)} and \textbf{(M3)} are valid, then by Proposition~\ref{pr:diss}
and Remark~\ref{re:dis}(1) there exists on a bounded forward invariant absorbing set.
Thus under the conditions of Proposition~\ref{pr:diss}
we can apply Theorem~\ref{th:qs} on this set.
Namely, we have the following assertion.
\begin{corollary}\label{co:qs}
Let conditions (A1), (F1)-(F4) and (M3) with \eqref{sdd-2nd-35a} be in force.
Let $\cB_0$ ba a forward invariant absorbing set for $(S_t,W)$
such that $\cB_0\subset \{\va\in W : |\va |_W\le R\}$. Then there exist  $C_i(R)>0$
and $\widetilde{\lambda}>0$ such that \eqref{qs-est} holds for any pair of solutions $u^1(t)$
and $u^2(t)$ starting from $\cB_0$.
\end{corollary}
\begin{remark}\label{re:qs}
  {\rm
  Taking in \eqref{qs-est}  maximum over the interval $[t-h,t]$
 yields
 \begin{align}\label{qs-est2}
| S_tu^1-S_tu^2|_W
\le   C_1(R) he^{\widetilde{\lambda} h} e^{-\widetilde{\lambda} t} |\va^1-\va^2|_W
+
C_2(R)h\max_{s\in [0,t]}\mu_{W} (u^1_s-u^2_s), ~~ t\ge h.
\end{align}
 where $\mu_{W} (\va)\equiv \left\{\max_{\theta\in [-h,0]}
||A^{{1\over 2}-\delta}\va(\theta)||\right\}$ is a compact semi-norm\footnote{We
recall  that a semi-norm $\widetilde{n}(x)$ on a Banach space $X$ is
said to be compact iff for any bounded set $B\subset X$ there exists
a sequence $\{x^n\}\subset B$ such that $\widetilde{n}(x^m - x^k)\to
0$ as $m,k\to \infty$. }
on $W$.
The quasi-stability property in \eqref{qs-est2} has the structure
  which is different from the standard form
  (see, e.g., \cite{Chueshov-Lasiecka-MemAMS-2008_book,Chueshov-Lasiecka-2010_book,CL-hcdte-notes})
  of quasi-stability inequalities for (non-delayed) second order in time equations.
  However as we will see below the consequences in our case are the same as in
  the case of standard quasi-stable systems.  We also note that quasi-stability
  properties in different forms were important in many situations
  in the long-time dynamics studies
  (see, e.g., the discussion in \cite[Remark 7.9.3]{Chueshov-Lasiecka-2010_book}).
  }
\end{remark}

 We split the proof of Theorem~\ref{th:qs} in two cases
 and start with the simplest one.

 \subsubsection*{Proof of Theorem~\ref{th:qs} in the subcritical case}
 We rely on the  mild solutions  form  \eqref{sdd-2nd-05a} of the problem
 and follow the line of argument  given in
\cite[p.479-480]{Chueshov-Lasiecka-2010_book} with modifications
necessary for the case of state dependent delay force $M$.
We note that similar to \cite[p.58-62]{Chueshov-Lasiecka-MemAMS-2008_book}
 we can also use here the multipliers method.
However for the completeness we   demonstrate here the constant variation
method. The multipliers method  is presented below in the case of the critical force $F$.
\smallskip\par
 Let us consider two solutions $U^1=(u^1, \dot u^1)$ and $ U^2=(u^2, \dot u^2)$ of
(\ref{sdd-2nd-01}), (\ref{sdd-2nd-IC})
possessing   (\ref{sdd-2nd-22a}). Using (\ref{sdd-2nd-05a}) and exponential stability of
the semigroup $e^{-\cA t}$ in the space
 $Y=D(A^{1/2})\times H$ we have
that
\begin{equation}\label{sdd-2nd-36}
 ||U^1(t)-U^2(t)||_Y \le e^{-\widetilde{\lambda} t} ||U^1(0)-U^2(0)||_Y  +\int^t_0 e^{-\widetilde{\lambda}(t-s) }
 ||{\cal N} (U^1_s)-{\cal N} (U^2_s)||_Y \, ds,\quad t>0,
\end{equation}
with $\wla>0$, where $\cN$ is given by \eqref{N-def}.
Since
\[
||{\cal N} (U^1_s)-{\cal N} (U^2_s)||_Y \le
||F(u^1(t))-F(u^2(t))|| + ||M(u^1_t)-M(u^2_t)||,
\]
 using
properties  \eqref{sdd-2nd-35} and (\ref{sdd-2nd-34})
we obtain
\[
||{\cal N} (U^1_s)-{\cal N} (U^2_s)||_Y \le
C(R)\max_{\theta\in [-h,0]} ||A^{{1\over
2}-\delta}(u^1(s+\theta)-u^2(s+\theta))||
\]
for some $\delta>0$.
Thus \eqref{sdd-2nd-36} yields
\begin{align}\label{sdd-2nd-36a}
 ||U^1(t)-U^2(t)||_Y \le & e^{-\widetilde{\lambda} t} ||U^1(0)-U^2(0)||_Y +C(R) I(t,u^1-u^2),\quad t>0,
 \end{align}
 where
\begin{equation*}
  I(t,z) =\int^t_0 e^{-\widetilde{\lambda}(t-s) }
 \max_{\ell\in [-h,0]} ||A^{{1\over 2}-\delta}z(s+\ell)|| \, ds ~~\mbox{with $z(s)=u^1(s)-u^2(s)$}.
\end{equation*}
Now we split $  I(t,z)$ as   $I(t,z)=  I^1(t,z)+I^2(t,z)$, where
\begin{align*}
I^1(t,z) \equiv & \int^h_0 e^{-\widetilde{\lambda} (t-s)}
\max_{\ell\in [-h,0]} ||A^{{1\over 2}-\delta}z(s+\ell)||\, ds \le
C_{R,h} |z_0|_W  \int^h_0 e^{-\widetilde{\lambda} (t-s)}\, ds
\\
= & C_{R,h} |z_0|_W  \cdot
e^{-\widetilde{\lambda} t} (e^{\widetilde{\lambda} h}
-1)\widetilde{\lambda}^{-1}
\end{align*}
and
\begin{align*}
I^2(t,z) \equiv & \int^t_h  e^{-\widetilde{\lambda} (t-s)}
\max_{\ell\in [-h,0]} ||A^{{1\over 2}-\delta}z(s+\ell)||\, ds
\\
\le & \int^t_0  e^{-\widetilde{\lambda} (t-s)} \max_{\xi\in [0,t]}
||A^{{1\over 2}-\delta}z(\xi)||\, ds =   (1-e^{-\widetilde{\lambda}
t}) \widetilde{\lambda}^{-1} \cdot  \max_{\xi\in [0,t]} ||A^{{1\over
2}-\delta}z(\xi)||.
\end{align*}
Thus  \eqref{sdd-2nd-36a} yields the desired estimate in \eqref{qs-est} for the subcritical  nonlinearity $F$.

 \subsubsection*{Proof of Theorem~\ref{th:qs} in the critical case
 with large damping}

We follow the line of  the arguments
of \cite[p. 85, Theorem 3.58]{Chueshov-Lasiecka-MemAMS-2008_book}. 
\smallskip\par

Let $u^1$ and $u^2$ be solutions satisfying  \eqref{sdd-2nd-22a}. Then
 $z=u^1-u^2$ solves the equation
\begin{equation}\label{sdd-2nd-44}
\ddot z(t) +A z(t) + k\dot z(t) = - F_{1,2}(t)- M_{1,2}(t)
\end{equation}
with
\[
F_{1,2}(t)\equiv F(u^1(t))-F(u^2(t)); ~~ M_{1,2}(t)\equiv  M(u^1_t)-M(u^2_t).
\]
We multiply the last equation by $\dot z(t)$ and integrate over
$[t,T]$:
\begin{equation}\label{sdd-2nd-45}
E_z(T) - E_z(t) + k \int^T_t ||\dot z(s)||^2\, ds = -\int^T_t
(F_{1,2}(s),\dot z(s))\, ds - \int^T_t (M_{1,2}(s),\dot z(s))\, ds.
\end{equation}
Here we denote $E_z (t)\equiv \hf(||\dot z(t)||^2 +
||A^{\hf} z(t)||^2)$.
\par
One can check that there is constant $C_R>0$ such that
\[
|(G_{1,2}(t),\dot z(t))| \le
\varepsilon ||A^{1\over 2} z(t)||^2 + \frac{C_R}{\varepsilon}
||\dot z(t)||^2,~~ \forall\, \eps>0.
\]
 Similarly, using assumption \textbf{(M4)}, we have
\[
|(M_{1,2}(t),\dot z(t))| \le \max_{\theta\in [-h,0]}
||A^{{1\over 2}-\delta} z(t+\theta)||^2 + C_R||\dot z(t)||^2.
\]
 Hence, we get from
(\ref{sdd-2nd-45})
\begin{multline}\label{sdd-2nd-46a}
\Big|E_z(T) - E_z(t) + k \int^T_t ||\dot z(s)||^2\, ds\Big| \\\
\le   \varepsilon \int^T_t ||A^{1\over 2} z(s)||^2\, ds +
\int^T_t \max_{\theta\in [-h,0]} ||A^{{1\over
2}-\delta} z(s+\theta)||^2\, ds + C_R\left(1+ {1\over
\varepsilon}\right) \int^T_t ||\dot z(s)||^2\, ds
\end{multline}
for every $\eps>0$.
Below we choose (assume that) $k$ is big enough to satisfy (see the
the last term  in  (\ref{sdd-2nd-46a}))
\begin{equation}\label{sdd-2nd-48}
C_R\left(1 + {1\over \varepsilon}\right) < {k\over
2}, \quad \mbox{ for all} \quad k\ge k_0.
\end{equation}
This choice is made for the simplification of the estimates only
(the final choice of $k_0$ to be done after the choice of
$\varepsilon$).
Now we multiply (\ref{sdd-2nd-44}) by $ z(t)$ and integrate over
$[0,T]$, using integration by parts. This yields
\begin{multline*}
(\dot z(T),z(T)) - (\dot z(0),z(0)) - \int^T_0 ||\dot z(s)||^2\,
ds + \int^T_0 ||A^{1\over 2} z(s)||^2\, ds + k\int^T_0 (\dot z(s), z(s))\, ds
\\
\le {1\over 2}\int^T_0 ||A^{1\over 2} z(s)||^2\, ds +  \widetilde{C_R} \int^T_0 || z(s)||^2\, ds +
\widetilde{C_R} \int^T_0 \max_{\theta\in [-h,0]} ||A^{{1\over
2}-\delta} z(s+\theta)||^2\, ds.
\end{multline*}
Hence, using the definition of
$E_z$ after (\ref{sdd-2nd-45})
and the relation
\[
 k\int^T_0 (\dot z(s), z(s))\, ds \le  \hf\int^T_0 \|\dot z(s)\|^2 \, ds+ \frac{k^2}2\int^T_0 \| z(s)\|^2 \, ds,
\]
we obtain that
\begin{align}\label{sdd-2nd-47}
\hf \int^T_0 ||A^{1\over 2} z(s)||^2\, ds \le & \frac32\int^T_0 ||\dot z(s)||^2\,
ds + C(E_z(0)+E_z(T)) \notag \\
& +
\widetilde{C_R}(k) \int^T_0 \max_{\theta\in [-h,0]}
||A^{{1\over 2}-\delta} z(s+\theta)||^2\, ds.
\end{align}
 From (\ref{sdd-2nd-46a}) with $t=0$ and using
(\ref{sdd-2nd-48}) we get
\begin{align}\label{sdd-2nd-49}
E_z(0) \le & E_z(T) + {3k\over 2} \int^T_0 ||\dot z(s)||^2\, ds +
\varepsilon \int^T_0 ||A^{1\over 2} z(s)||^2\, ds
\notag
\\ &+
\int^T_0 \max_{\theta\in [-h,0]} ||A^{{1\over
2}-\delta} z(s+\theta)||^2\, ds.
\end{align}
It follows from (\ref{sdd-2nd-46a}) with help of integration over $[0,T]$ (we use
(\ref{sdd-2nd-48}) again) that
\begin{equation}\label{sdd-2nd-50}
T E_z(T)\le \int^T_0 E_z(s)\, ds  +  \varepsilon T \int^T_0 ||A^{1\over 2} z(s)||^2\,
ds + T \int^T_0 \max_{\theta\in [-h,0]} ||A^{{1\over
2}-\delta} z(s+\theta)||^2\, ds.
\end{equation}
Another consequence of (\ref{sdd-2nd-46a}) for $t=0$, using
(\ref{sdd-2nd-48}), is
\begin{equation}\label{sdd-2nd-51}
{k\over 2}\int^T_0 ||\dot z(s)||^2\, ds \le E_z(0) + \varepsilon
\int^T_0 ||A^{1\over 2} z(s)||^2\, ds + \int^T_0
\max_{\theta\in [-h,0]} ||A^{{1\over 2}-\delta} z(s+\theta)||^2\,
ds.
\end{equation}
Considering the sum of (\ref{sdd-2nd-51}) and (\ref{sdd-2nd-47})
and assuming that $k\ge8$
 we can get
\begin{align}\label{sdd-2nd-52}
{k}\int^T_0 ||\dot z(s)||^2\, ds + \int^T_0 E_z (s)\, ds \le & C(E_z(0)+E_z(T))
 + 4\varepsilon \int^T_0 ||A^{1\over 2}
z(s)||^2\, ds\notag  \\ &
+ C^*_{R,k}\int^T_0 \max_{\theta\in [-h,0]}
||A^{{1\over 2}-\delta} z(s+\theta)||^2\, ds.
\end{align}
Now we add to the both sides of (\ref{sdd-2nd-52}) the value
${1\over 2} T E_z(T)$ and use (\ref{sdd-2nd-50})
\begin{multline}\label{sdd-2nd-53}
  {k}\int^T_0 ||\dot z(s)||^2\, ds + \hf \int^T_0 E_z (s)\, ds + {1\over 2} T E_z(T)
\\ \le 4 \varepsilon (1+T)
\int^T_0 ||A^{1\over 2} z(s)||^2\, ds  +  C(E_z(0)+E_z(T))  \\
 + C_{R,k}(1+T)\int^T_0
\max_{\theta\in [-h,0]} ||A^{{1\over 2}-\delta} z(s+\theta)||^2\,
ds.
\end{multline}
Now we evaluate $E_z(0)+E_z(T) $. Using (\ref{sdd-2nd-49}) we have that
\begin{align*}
  E_z(0)+E_z(T) \le & 2E_z(T) + {3k\over 2} \int^T_0 ||\dot z(s)||^2\, ds +
\varepsilon \int^T_0 ||A^{1\over 2} z(s)||^2\, ds    \\ & +
\int^T_0 \max_{\theta\in [-h,0]} ||A^{{1\over
2}-\delta} z(s+\theta)||^2\, ds.
\end{align*}
Substituting this  into (\ref{sdd-2nd-53}) we get that
\begin{multline*}
 {1\over 2}\int^T_0 E_z (s)\, ds + \left({1\over 2} T-2C\right) E_z(T)
 \le c_0k \int^T_0 ||\dot z(s)||^2\, ds\\
+  c_1 \varepsilon \left(1+T\right) \int^T_0 ||A^{1\over 2}
z(s)||^2\, ds +
 \widetilde{C_R}(k)(1+T)\int^T_0 \max_{\theta\in
[-h,0]} ||A^{{1\over 2}-\delta} z(s+\theta)||^2\, ds
\end{multline*}
Assuming that
\begin{equation}\label{sdd-2nd-54}
{1\over 2} T-2C>1,
\end{equation}
we get
\begin{multline}\label{sdd-2nd-55}
  E_z(T) + {1\over 2}\int^T_0 E_z (s)\, ds
\le  C_1 \varepsilon \left(1+T\right) \int^T_0 ||A^{1\over 2} z(s)||^2\, ds
\\
+ \left(1+T \right) \widetilde{C_R}(k)\int^T_0 \max_{\theta\in [-h,0]}
||A^{{1\over 2}-\delta} z(s+\theta)||^2\, ds + \widetilde{c_0} k
\int^T_0 || \dot z(s)||^2\, ds.
\end{multline}
To estimate the last term in (\ref{sdd-2nd-55}) we use
(\ref{sdd-2nd-46a}) with $t=0$ (remind (\ref{sdd-2nd-48})) to get
\begin{equation*}
{k\over 2} \int^T_0 ||\dot z(s)||^2\, ds \le E_z(0) - E_z(T) +
\varepsilon \int^T_0 ||A^{1\over 2} z(s)||^2\, ds +
\int^T_0 \max_{\theta\in [-h,0]} ||A^{{1\over
2}-\delta} z(s+\theta)||^2\, ds.
\end{equation*}
So, we can rewrite (\ref{sdd-2nd-55}) as
\begin{align}\label{sdd-2nd-57}
E_z(T)& + {1\over 2}\int^T_0 E_z (s)\, ds \le 2c_0(E_z(0) - E_z(T))+
\notag \\
 & C_1 \varepsilon \left(1+T\right) \int^T_0 ||A^{1\over 2} z(s)||^2\, ds +
\left(1+T \right) \widetilde{C_R}(k)\int^T_0 \max_{\theta\in
[-h,0]} ||A^{{1\over 2}-\delta} z(s+\theta)||^2\, ds.
\end{align}
Since $||A^{1\over 2} z(s)||^2 \le 2 E_z(s) $, the choice of small
$\varepsilon>0$ to satisfy
\begin{equation}\label{sdd-2nd-58}
C_1 \varepsilon \left(1+T\right) < {1\over 4}
\end{equation}
simplifies (\ref{sdd-2nd-57}) as follows
$$ E_z(T) \le  \widetilde{c_0}(E_z(0) - E_z(T)) + \left(1+T\right) \widetilde{C_R}(k)\int^T_0
\max_{\theta\in [-h,0]} ||A^{{1\over 2}-\delta} z(s+\theta)||^2\,
ds. $$ The last step is
$$
E_z(T) \le {\widetilde{c_0}\over 1+\widetilde{c_0}} E_z(0) +
\widetilde{C_R}(T,k)\int^T_0 \max_{\theta\in [-h,0]} ||A^{{1\over
2}-\delta} z(s+\theta)||^2\, ds.
$$
Since  $\gamma\equiv {\widetilde{c_0}\over 1+\widetilde{c_0}} < 1$
this  means that there is $w>0$ such that
\begin{equation}\label{sdd-2nd-59}
E_z(T) \le e^{-w T} E_z(0) + C_{R,T,k} \int^T_0 \max_{\theta\in
[-h,0]} ||A^{{1\over 2}-\delta} z(s+\theta)||^2\, ds.
\end{equation}
We mention that the parameters were chosen in the following order.
First we choose $T>h$ to satisfy (\ref{sdd-2nd-54}), next we choose
small $\varepsilon>0$ to satisfy (\ref{sdd-2nd-58}) and finally we
choose $k$ big enough to satisfy (\ref{sdd-2nd-48}).

Now  using the same step by step procedure ($mT\mapsto (m+1)T)$ as in the
Remark 3.30 \cite{Chueshov-Lasiecka-MemAMS-2008_book}
we can derive  the conclusion in \eqref{qs-est}
 from the relation in \eqref{sdd-2nd-59} written on the interval $[mT,(m+1)T]$.
 Thus the proof of Theorem \ref{th:qs} is complete.

\section{Global  and exponential attractor}\label{sect:attr}

In this section relying on Proposition~\ref{pr:diss} and Theorem~\ref{th:qs}
we establish the existence of
a global attractor and study its properties.
We recall
(see, e.g., \cite{Babin-Vishik,Chueshov_Acta-1999_book,Temam_book})
that
a \textit{global attractor}  of the  dynamical system  $(S_t, W)$
is defined as a bounded closed  set $\fA\subset W$
which is  invariant ($S(t)\fA=\fA$ for all $t>0$) and  uniformly  attracts
all other bounded  sets:
$$
\lim_{t\to\infty} \sup\{{\rm dist}_W(S(t)y,\fA):\ y\in B\} = 0
\quad\mbox{for any bounded  set $B$ in $W$.}
$$
We note (see, e.g., \cite{Temam_book}) that the global attractor consists of bounded
full trajectories. In the case of the delay system $(S_t,W)$ a full trajectory
can be described as a function $u$ from   $C(\Rset, D(A^{1/2}))\cap C^1(\Rset, H)$
possessing the property $S_tu_s=u_{t+s}$ for all $s\in \Rset$, $t\ge 0$.

The main consequence of dissipativity and quasi-stability given by Proposition~\ref{pr:diss} and Theorem~\ref{th:qs}
is the following theorem.

\begin{theorem}[Global Attractor]\label{th:attractor}
 Let assumptions (A1) and (F1)-(F4)  be in force.
Assume that  the term $M(u_t)$ has form \eqref{sdd-2nd-26}
with $\tau: W \mapsto [0,h]$ possessing property \eqref{sdd-2nd-35a}.
Then  the dynamical system $(S_t,W)$
generated by \eqref{sdd-2nd-01}
possesses the compact global attractor $\fA$ of
finite fractal dimension
\footnote{
For the definition and some  properties of
the {\em fractal dimension}, see, e.g., \cite{Chueshov_Acta-1999_book} or \cite{Temam_book}.
}.
Moreover,
 for any full trajectory $\{u(t)\, : t\in\Rset\}$
such that $u_t\in\fA$ for all $t\in\Rset$ we have that
 \begin{equation}\label{attr-sm}
    \ddot{u}\in L_\infty(\Rset, H),~~ \dot{u}\in L_\infty(\Rset, D(A^{1/2}))
    ~~ u\in L_\infty(\Rset, D(A))
 \end{equation}
 and
 \begin{equation}\label{attr-sm2}
    \|\ddot{u}(t)\|+ \|A^{1/2}\dot{u}(t)\|+  \|A u(t)\|\le R_*,~~~ \forall\, t\in\Rset.
 \end{equation}
Under the hypotheses of Corollary~\ref{co:smooth} we also have that $\fA$ is a bounded set in
in $W_{sm}$ and lies in
$\cL$,
where $W_{sm}$ and $\cL$ are given by \eqref{Wsm} and \eqref{L-smoth}.
\end{theorem}
\begin{proof}

Since the system $(S_t,W)$ is dissipative (see Proposition~
\ref{pr:diss}) for the existence of a compact global attractor we need to
prove that $(S_t,W)$ is asymptotically smooth.\footnote{
According \cite{Ha88}
this means that for any bounded forward invariant set $B$ in $W$
there exists a compact set $K$ in $W$ which attracts uniformly
$S_tB$ as $t\to+\infty$.
}
 For this we can use the
 Ceron-Lopes type criteria (see, e.g.,  \cite{Ha88} or
\cite{Chueshov-Lasiecka-MemAMS-2008_book})
which in fact states (see \cite[p.19, Corollary 2.7]{Chueshov-Lasiecka-MemAMS-2008_book})
that the quasi-stability  estimate in \eqref{qs-est2} implies that
 $(S_t,W)$  is an asymptotically smooth dynamical system.
Thus
the existence of a compact global attractor  is established.
\par

To get the finite dimensionality of the attractor we apply the same idea  as in
\cite{Chueshov-Lasiecka-MemAMS-2008_book} and \cite{Chueshov-Lasiecka-2010_book}
which is originated from the
M\'{a}lek--Ne\v{c}as method of "short" trajectories (see
\cite{malek-ne,malek}).
However we use a completely different choice of the space
of "short" trajectories which is motivated by the delay structure of the model and
the choice of the phase space.
\par
As in \cite{Chueshov-Lasiecka-MemAMS-2008_book,Chueshov-Lasiecka-2010_book}
 we
rely on  the abstract result  \cite[Theorem
2.15, p.23]{Chueshov-Lasiecka-MemAMS-2008_book} on finite
dimensionality of bounded closed sets in a Banach space which are invariant with respect to
a Lipschitz mapping  possessing some squeezing property.
We consider
the auxiliary space
\begin{equation*}
 W(-h,T) \equiv  C([-h,T]; D(A^{1\over 2})) \cap C^1([-h,T]; H),~~ T>0,
\end{equation*}
endowed with the norm
\[
|\varphi|_{W(-h,T)}=\max_{s\in [-h,T]} ||A^{1/2} \varphi(s)||+
  \max_{s\in [-h,T]} ||\dot \varphi(s)||.
\]
We note that in the case $T=0$ we have $W(-h,0)=W$. Thus $W(-h,T)$
is the space of extensions  with the same smoothnes of functions from $W$ on the interval
$[-h,T]$.
\par
Let $\sB$ be a set in the phase space $W$. We denote by $\sB_T$ the
set of functions $u\in W(-h,T)$ which solve \eqref{sdd-2nd-01} with
initial data $u_{t\in[-h,0]}=\psi\in \sB$. We interpret $\sB_T$ as a
set of "pieces" of trajectories starting from $\cB$.
 We also define the shift
 (along solutions to  (\ref{sdd-2nd-01})) operator $\sR_T\, :\sB_T\mapsto  W(-h,T)$ by the
formula
\begin{equation}\label{sdd-2nd-RT}
(\sR_Tu)(t)= u(T+t),~~ t\in [-h,T],
\end{equation}
where
$u$ is the solution to (\ref{sdd-2nd-01}) with initial data from $\sB$.
\par
The following lemma states that the mapping $\sR_T$ satisfies
some contractive property modulo compact terms.

\begin{lemma}\label{le:V-qs}
Let $\sB$ be a forward  invariant set for the dynamical system
$(S_t,W)$ such that $\sB\in\{ \phi\, : |\phi|_W\le R\}$ for some
$R$. Let $T>h$. Then $\sB_T$ is forward invariant with respect  to
the shift operator  $\sR_T$ and
\begin{align}
    \label{sdd-2nd-dim1}
|\sR_T\varphi^1 - \sR_T\varphi^2|_{W(-h,T)} \le & c_1(R)e^{-
\widetilde{\lambda}(T-h)} |\varphi^1 - \varphi^2|_{W(-h,T)}
 \notag \\
 &+ c_2(R)\left[  n(\varphi^1 - \varphi^2)+ n(\sR_T\varphi^1 - \sR_T\varphi^2)\right]
\end{align}
for every $\va^1,\va^2\in\sB_T$, where $n(\va)=\sup_{s\in [0,T]}
||A^{{1\over 2}-\delta}\va(s)||$ is a compact seminorm  (see the footnote in Remark ~\ref{re:qs}
  for the definition) on  the space $W(-h,T)$.
\end{lemma}
\begin{proof}
 The invariance of $\sB_T$ is obvious due to the construction.
 The relation in \eqref{sdd-2nd-dim1} follows from Theorem~\ref{th:qs}.
 The compactness of the seminorm $n$ is implied by the
 infinite dimensional version of Arzelа--Ascoli  theorem,
 see the Appendix in \cite{Chueshov-Lasiecka-2010_book}, for instance.
\end{proof}

We choose  $T>h$ such that
 $\eta_T=c_1(R)e^{- \widetilde{\lambda} (T-h)}  <1$ and take $\sB=\fA$,
 where $\fA$ is the global attractor.
It is clear that
the set $\fA_T$ is strictly invariant. Therefore we can apply
\cite[Theorem 2.15, p.23]{Chueshov-Lasiecka-MemAMS-2008_book} to get
the finite dimensionality of the set $\fA_T$ in $W(-h,T)$. The finial step is to
consider the restriction mapping
\[
r_h\, : \{u(t), t\in [-h,T]\}\mapsto  \{u(t), t\in [-h,0]\}
\]
which is obviously Lipschitz continuous from $W(-h,T)$ into $W$.
Since $r_h\fA_T=\fA$ and  Lipschitz mappings do not increase
fractal dimension of a set, we conclude that
 $$
 \dim^W_f {\fA} \le \dim^{W(-h,T)}_f {\fA}_T <\infty.
 $$
 \par
 To prove the regularity properties in \eqref{attr-sm} and \eqref{attr-sm2}
 we can use Theorem~\ref{th:qs} and the same
 idea as in \cite{Chueshov-Lasiecka-MemAMS-2008_book,Chueshov-Lasiecka-2010_book},
 see  also \cite{CL-hcdte-notes}. Indeed,
  let $\ga=\{ u(t)\, :\; t\in\Rset\}$ be a full trajectory
of the system,  i.e., $(S_tu_s)(\theta)=u(t+s+\theta)$ for $\theta\in [-h,0]$.
Assume that $u_t\in\fA$ for all $t\in\Rset$.
Consider the difference of this trajectory and its small shift
  $\ga_\varepsilon=\{ u(t+\varepsilon)\, :\; t\in\Rset\}$ and apply the inequality
  in \eqref{qs-est} with starting point at $s\in\Rset$:
\begin{align*}
||\dot u(t+\varepsilon)-\dot u(t)||^2 + ||A^{1\over
2}(u(t+\varepsilon)- u(t))||^2
\le & C_1(R)  e^{-\widetilde{\lambda} (t-s)} |u_{s+\varepsilon}-u_{s}|^2_W \notag \\
& + C_2(R)\max_{\xi\in [s,t]}
||A^{{1/2}-\delta}(u(\xi+\varepsilon)-u(\xi))||^2.
\end{align*}
 Since $u_s\in \fA$ for all $s\in\Rset$, in the limit $s\to-\infty$ we obtain that
  \begin{align*}
||\dot u(t+\varepsilon)-\dot u(t)||^2 + ||A^{1\over
2}(u(t+\varepsilon)- u(t))||^2 \le  C_2(R)\sup_{\xi\in [-\infty,t]}
||A^{{1/2}-\delta}(u(\xi+\varepsilon)-u(\xi))||^2.
\end{align*}
Now in the same way as in
\cite[p.102,103]{Chueshov-Lasiecka-MemAMS-2008_book} or in
\cite[p.386,387]{Chueshov-Lasiecka-2010_book}
We can conclude that
 \[
 \frac1{\varepsilon^2}\left[ ||\dot u(t+\varepsilon)-\dot u(t)||^2 + ||A^{1\over 2}(u(t+\varepsilon)- u(t))||^2
 \right]
  \]
  is uniformly bounded in $\varepsilon\in (0,1]$. This implies (passing with the limit $\varepsilon\to 0$) that
\[
  ||\ddot u(t)||^2 + ||A^{1\over 2}\dot u(t)||^2\le C_R.
  \]
  Now using  equation \eqref{sdd-2nd-01} we conclude that  $||A u(t)||^2\le C_R$.
  This gives (\ref{attr-sm}) and (\ref{attr-sm2}).
  \par
  The final statement follows from Corollary~\ref{co:smooth} and Remark~\ref{re:smooth}.
  \par
 This completes the proof of Theorem~\ref{th:attractor}.
 \end{proof}

Now  we present a result on the existence of fractal exponential
attractors. We recall the following definition.

\begin{definition}[cf. \cite{EFNT94}]\label{de7.3.2}
A compact set $\fA_{\rm exp}\subset W$
is said to be (generalized)  fractal exponential
attractor
for  the dynamical system $(S_t,W)$  iff  $\fA$ is a positively invariant set
whose fractal dimension is finite (in some extended space  $\cW\supset W$)  and
for every bounded set $D\subset W$ there exist positive constants
$t_D$, $C_D$ and $\gamma_D$ such that
\begin{equation}\label{7.3.4}
d_W\{S_tD\, |\, A_{\rm exp}\}\equiv \sup_{x\in D} \mbox{dist}\,_W
(S_tx,\, \fA_{\rm exp})\le C_D\cdot e^{-\gamma_D(t-t_D)}, \quad t\ge
t_D.
\end{equation}
\end{definition}
This  concept has been introduced in \cite{EFNT94} in the case when $\cW$ and $W$ are the same.
For  details  concerning fractal exponential attractors
we refer to \cite{EFNT94} and also to recent survey \cite{MirZel-08}.
We only mention that (i) a global attractor can be non-exponential and (ii) an exponential attractor
is not unique and   contains the global attractor.
\par

Using the quasi-stability estimate and ideas presented in \cite{Chueshov-Lasiecka-MemAMS-2008_book,Chueshov-Lasiecka-2010_book}
we can construct fractal exponential attractors for the system considered.

\begin{theorem}\label{th:exp-attr}
Let the hypotheses of Theorem~\ref{th:attractor}   be in force.
Then the dynamical system $(S_t,W)$ possesses a (generalized) fractal exponential attractor
whose dimension is finite in the space
\[
\cW \equiv  C([-h,T]; D(A^{{1\over 2}-\delta})) \cap C^1([-h,T]; H_{-\delta}),~~~\forall\, \delta>0,
\]
where $H_{-s}$, $s>0$, denotes the closure of $H$ with respect   to
the norm $\|A^{-s}\cdot\|$.
\end{theorem}
\begin{proof}
Let $\cB$ be a forward invariant bounded absorbing set for
$(S_t,W)$ which exists due to Proposition~\ref{pr:diss} and Remark~\ref{re:dis}(1).
Then we apply Lemma~\ref{le:V-qs} to obtain (discrete)
quasi-stability property for the shift mapping $\sR_T$ defined in (\ref{sdd-2nd-RT}) on $\B_T$.
We choose  $T>h$ in \eqref{sdd-2nd-dim1} such that $\eta_T=c_1(R)e^{-\widetilde{\lambda} (T-h)}  <1$
and apply \cite[Corollary
2.23]{Chueshov-Lasiecka-MemAMS-2008_book}
which gives us that the mapping ${\cal R}_T$,
 possesses a fractal exponential
attractor ${\cal A}_T$.
Next, using \eqref{sdd-2nd-01} we can see that $\|\ddot
u(t)\|_{-2}<C_R$ for all $t\in\Rset$. This allows us to show that
$S_t\va$ is a H\"{o}lder continuous in $t$ in the space $\sW$, i.e.,
\begin{equation}\label{Hold}
|S_{t_1}\va-S_{t_2}\va|_{\cW}\le C_{\cB} |t_1-t_2|^\gamma,\quad
t_1,t_2\in \Rset_+,\; y\in\cB,
\end{equation}
 for some positive $\ga>0$.
 Now we consider the restriction map $r_h$
(see above) and the sets $r_h{\cal A}_T = {\cal A}\subset W$,
$A_{exp}\equiv \bigcup \{ S_t {\cal A} : t\in [0,T]\}\subset W$.
It is clear that $A_{exp}$ is forward invariant.
Since $r_h$ is Lipschitz from $W(-h,T)$ into $W$, $\cA$ is finite-dimensional. Therefore
the property in \eqref{Hold} implies that $A_{exp}$ has a finite fractal dimension
in $\cW$. As in \cite[p.123]{Chueshov-Lasiecka-MemAMS-2008_book} we can see
that $A_{exp}$ is an  exponentially
attracting set  for $(S_t,W)$. This completes the proof  of Theorem~\ref{th:exp-attr}.
\end{proof}

In conclusion of this section we note that using quasi-stability
property (\ref{qs-est}) we can also establish some other asymptotic
properties the system $(S_t,W)$. For instance,
in the same way as it is done in
\cite{Chueshov-Lasiecka-MemAMS-2008_book} and \cite{Chueshov-Lasiecka-2010_book}
we can suggest criteria which guarantee the existence of finite number of determining functionals.

\section{Examples}\label{sect:ex}
In this section we discuss several possible applications of the results above.

\subsection{Plate models}\label{sect:ex-plate}
Our main applications are related to nonlinear plate models.
\par
Let $\Om\subset \Rset^2$ be a bounded smooth domain.
In the space $H=L_2(\Om)$ we consider the following  problem
\begin{subequations}\label{plate-model}
\begin{align}
&\pd_{tt} u(t,x)+ k\pd_t u(t,x) +\Delta^2 u(t,x) + F(u(t,x)) +a u(t-\tau[u_t],x) = 0, ~~x\in\Om,~ t>0,
\label{app-1}
\\ &u =\frac{\pd u}{\pd n}=0~~\mbox{on}~~\pd\Om, ~~~ u(\theta)=\va(\theta)~~\mbox{for}~~\theta\in
[-h,0].
\label{app-2}
\end{align}
 \end{subequations}
We assume that $\tau$ is a continuous mapping from $C(-h,0;
H^2_0(\Om))\cap C^1(-h,0; L_2(\Om))$ into the interval $[0,h]$. As
it was  already mentioned in  Introduction the delay term in
\eqref{app-1} models the reaction of foundation.
\par
The model in \eqref{plate-model} can be written in the abstract form  \eqref{sdd-2nd-01}
with $A = \Delta^2$ defined on the  domain $D(A)=H^4\cap H^2_0(\Omega)$.
Here and below $H^s(\Om)$ is the Sobolev space of the order $s$
and  $H^s_0(\Om)$ is the closure of $C_0^\infty(\Om)$ in  $H^s(\Om)$.
In this case we have
$D(A^{s})=H^{4s}_0(\Om)$ for $0\le s\le 1/2$, $s\neq 1/8, 3/8$.
\par
As the simplest example of delay terms satisfying all hypotheses in \textbf{(M1)}--\textbf{(M4)} we can consider
\begin{equation}\label{tau-g}
    \tau[u_t]= g(Q[u_t]),
\end{equation}
where $g$ is a smooth mapping from $\Rset$ into $[0,h]$ and
\[
Q[u_t]=\sum_{i=1}^Nc_i u(t-\sigma_i, a_i).
\]
Here $c_i\in \Rset$, $\si_i\in [0,h]$, $a_i\in\Om$ are arbitrary
elements. We could also consider the term $Q$ with the Stieltjes integral
over delay interval $[-h,0]$ instead of the sum. Another possibility
is to consider combination of averages like
\begin{equation}\label{Q-aver}
  Q[u_t]=\sum_{i=1}^N  \int_{\Om}u(t-\sigma_i, x)\xi_i(x) dx,
\end{equation}
where   $\si_i\in [0,h]$ and $\{\xi_i\}$ are arbitrary functions from $L_2(\Om)$. We can also consider linear combinations of
these Q's and also their powers and products. The corresponding
calculations are simple and related to the fact that for every
$s>1/4$ the space $D(A^s)$ is an algebra belonging to
$C(\overline{\Om})$.
\par
As for nonlinearities $F$
satisfying  requirements \textbf{(F1)}--\textbf{(F4)}
they are the same as in \cite{Chueshov-Lasiecka-MemAMS-2008_book}
and \cite{Chueshov-Lasiecka-2010_book}.
Therefore  delay perturbations of the models considered in these
sources in the case of linear damping provides us with a series
of examples. Here we only mention three of them.

\medskip\par \noindent
 {\bf Kirchhoff model}: In this case
$\cF(u)=  f(u)-h(x)$,
 where   $h\in L_2(\Om)$, and
 \begin{equation}
    \label{phi_condition-r}
f\in {\rm Lip_{loc}}(\Rset)~~ \mbox{ satisfies }  ~~  \underset{|s|\to\infty}{\liminf}\, f(s)s^{-1}=\infty.
\end{equation}
This is a subcritical case (see assumption \textbf{(F4)}, (\ref{sdd-2nd-34})
with $\eta>0$). The growth condition in \eqref{phi_condition-r} is needed to satisfy \eqref{sdd-2nd-20}
in \textbf{(F3)}.

\medskip\par
The following two examples are critical (assumption \textbf{(F4)},
(\ref{sdd-2nd-34}) with $\eta=0$).
\medskip\par \noindent
 {\bf Von Karman model:} In this model  (see, e.g., \cite{Chueshov-Lasiecka-2010_book,Lions})
$F(u)=-[u, v(u)+F_0]-h(x)$, where
 $F_0\in H^4(\Om) $ and $h\in L_2(\Om)$ are given functions,
\begin{equation*}
[u,v] = \partial ^{2}_{x_{1}} u\cdot \partial ^{2}_{x_{2}} v +
\partial ^{2}_{x_{2}} u\cdot \partial ^{2}_{x_{1}} v -
2\cdot \partial _{x_{1}x_{2}} u\cdot \partial _{x_{1}x_{2}} v,
\end{equation*} and
the  function $v(u) $ satisfies the equations:
\begin{equation*}
\Delta^2 v(u)+[u,u] =0 ~~{\rm in}~~  \Omega,\quad \frac{\pd v(u)}{\pd
n} = v(u) =0 ~~{\rm on}~~  \pd\Om.
\end{equation*}
For details concerning properties \textbf{(F1)}--\textbf{(F4)}
we refer to \cite[Chapter 6]{Chueshov-Lasiecka-MemAMS-2008_book}
and
 \cite[Chapters 4,9]{Chueshov-Lasiecka-2010_book}.
\medskip\par\noindent
 {\bf Berger Model:} In this case $   F(u)=
 \Pi'(u)$, where
\[
\Pi(u) =
 \frac{\kappa}4 \left[\int_\Om |\nabla u|^2 dx' \right]^2 -\frac{\mu}2 \int_\Om |\nabla u|^2 dx'
- \int_\Om u(x') h(x') dx',
\]
where $\kappa>0$ and $\mu\in\Rset$ are parameters,  $h\in L_2(\Om)$.
The analysis presented in  \cite[Chapter 4]{Chueshov_Acta-1999_book} and \cite[Chapter 7]{Chueshov-Lasiecka-MemAMS-2008_book}
yields the assumptions in \textbf{(F1)}--\textbf{(F4)}.

\subsection{Wave model}\label{sect:ex-wave}

Let $\Omega\subset \Rset^n$, $n=2,3$,  be a bounded domain
   with a sufficiently  smooth boundary $\Gamma$. The exterior normal on
  $\Gamma$ is denoted by $\nu$. We consider the following wave equation
\begin{equation*}
  \pd_{tt} u - \Delta u + k \pd_t u +    f(u) +u(t-\tau[u_t])= 0
 ~~ \mbox{ in }~~
Q=[0,\infty) \times \Omega
\end{equation*}
subject to  boundary condition
either of  Dirichlet type
\begin{equation}\label{dir}
 u=0  ~~\mbox{ on }~~ \Sigma\equiv [0,\infty) \times \Gamma,
\end{equation}
or else of Robin type
\begin{equation}\label{neu}
\partial_\nu u + u = 0  ~~\mbox{ on }~~
\Sigma.
\end{equation}
The initial conditions are given by $u(\theta)=\varphi(\theta), \,
\theta\in [-h,0]$.
In this case $H=L_2(\Om)$ and
$A$ is $-\Delta$ with either the Dirichlet \eqref{dir} or the Robin \eqref{neu}
boundary conditions.
So $D(A^{1/2})$ is either $H^1_0(\Om)$ or   $H^1(\Om)$ in this case.
\par
We assume that $k$ is a positive parameter and  the function
$f \in C^2(\Rset)$ satisfies the following polynomial growth
condition: there exists a positive constant $M > 0 $ such that
\begin{equation*}
|f''(s)| \leq M (1+ |s|^{q-1}) ,
\end{equation*}
where  $q \leq 2 $ when $n =3$ and $ q < \infty $
when $n =2$.
Moreover, we assume the same lower growth condition as \eqref{phi_condition-r}.
One  can see that the hypotheses in \textbf{(F1)}--\textbf{(F4)} are satisfied (see \cite[Chapter 5]{Chueshov-Lasiecka-MemAMS-2008_book}
for the detailed discussion). Moreover we have the subcritical case if $n=2$ or $n=3$ and $q<2$.
The case $n=3$ and $q=2$ is critical.
\par
As for the delay term $u(t-\tau[u_t])$
we can assume that, as in the plate models above, $\tau[u_t]$
has the form \eqref{tau-g} with $Q[u_t]$ given by \eqref{Q-aver}.
Moreover, instead of the averaging we can consider  an arbitrary family of linear functionals
on $H^{1-\delta}(\Om)$ for some $\delta>0$, i.e., we  can take
\[
Q[u_t]=\sum_{i=1}^Nc_i l_i[u(t-\sigma_i)],
\]
where $c_i\in \Rset$, $\si_i\in [0,h]$ and $l_i\in [H^{1-\delta}(\Om)]'$ are arbitrary elements.

\subsection{Ordinary differential equations}\label{ode}

The results above can be also applied in the ODE case when
$H=\Rset^n$, $A$ ia a symmetric  $n\times n$ matrix $A$ and the
nonlinear mappings $F : \Rset^n \to \Rset^n$, $M : C([-h,0];\Rset^n)
\to \Rset^n$ obey  appropriate requirements. The  space of initial states becomes
$W=C^1([-h,0];\Rset^n)$ (c.f.\ (\ref{sdd-2nd-05})) and
hence possesses a linear structure.
\par
Thus in contrast  with the solution manifold
suggested in \cite{Walther_JDE-2003} (see also \cite{Hartung-Krisztin-Walther-Wu-2006})
our approach do not  assume any nonlinear compatibility conditions
and provides us with a well-posedness result in a \emph{linear} phase space.
In addition,  both approaches produce the same class of solutions after some time.
To illustrate this effect we  consider  the  same
second order  delay ODE as it was used  in \cite{Walther_JDE-2003}
as a motivating example:
 \begin{subequations}\label{w1w2}
\begin{align}
&\dot u=v,~~ \dot v+ k  v = f(c s(u_t)-w), ~~ t>0,
\label{w1}
\\ &u(\theta) =\va^0(\theta),~~ v(\theta)=\va^1(\theta)~~\mbox{for}~~\theta\in
[-h,0].
\label{w2}
\end{align}
 \end{subequations}Here
$k$, $c$ and $w$ are positive reals, $s$ is a state-dependent delay
(implicitly defined in  \cite{Walther_JDE-2003}), $f:\Rset\to \Rset$
is a smooth function (for more details see
\cite[pp.61-64]{Walther_JDE-2003}). In the model $u$ is a position
of a moving object and $v$ is its velocity. The result of
\cite{Walther_JDE-2003} applied to this system says that if the
initial data $(\va^0;\va^1)$ belong to $C^1([-h,0];\Rset^2)$ and
satisfy the compatibility condition 
\begin{equation}\label{w3}
    \dot\va^0(0)=\va^1(0),~~\dot \va^1(0)+ k \va^1(0) = f(c s(\va^0)-w),
\end{equation}
then \eqref{w1w2} generates (local) $C^1$-semiflow on the solution
manifold
\[
{\cal M} =\left\{ (\va^0;\va^1) \in C^1([-h,0];\Rset^2)\, :~~\mbox{(\ref{w3}) is satisfied}\right\}.
\]
Application of our Theorem~\ref{th:well-pos} to the same system (written as a second order equation
with respect to $u$) says that if
the initial data
$(\va^0;\va^1)$ belong to $C^1([-h,0];\Rset)\times C([-h,0];\Rset)$
and are compatible in the natural way (as a position and the velocity):  $\dot\va^0(\theta)=\va^1(\theta)$
for all $\theta\in [-h,0]$, then under the same conditions as in \cite{Walther_JDE-2003}
we can avoid the (nonlinear) compatibility in \eqref{w3} and construct a  local semiflow in the space
\[
\widetilde{W} =\left\{ (\va^0;\va^1)\, :~ \dot\va^0(\theta)=\va^1(\theta)
~~\mbox{for all $\theta\in [-h,0]$,}~~
\va^0\in C^1([-h,0];\Rset)
\right\}.
\]
Thus we obtain another well-posedness class for the model in
\eqref{w1w2}. Moreover, by Corollary~\ref{co:smooth} the
corresponding solution $(u(t);v(t))$ is $C^1$ for $t\ge 0$ and
satisfies \eqref{w3} for $t>h$. Hence after time $t>h$ solutions
arrive at the same solution manifold $\cM$ as in
\cite{Walther_JDE-2003}. Similarly, starting at $\cM$ after time
$t>h$ we obviously arrive at $\widetilde{W}$ (see the first equation
in \eqref{w1}). Thus both classes of initial functions
$\widetilde{W}$ and ${\cal M}$ lead to exactly the same class of
solutions for $t>h$.
\par
As a bottom line we emphasize that in the case the second order
delay equations,
the natural (linear)
``position-velocity" compatibility provides us with an alternative  point of view on dynamics  and leads
to a simpler well-posedness argument
comparing to the method of a solution manifold presented in
\cite{Walther_JDE-2003}.

\medskip

{\bf Acknowledgments.}  This work was supported in part by GA
 CR under project P103/12/2431.

\end{document}